\documentclass[reqno, dvipsnames]{amsart}
\usepackage[utf8]{inputenc}
\usepackage{amscd, amssymb, amsthm, mathrsfs, amsmath, amsrefs, amsfonts, graphicx}
\usepackage[all]{xy}
\usepackage{pifont}
\usepackage[colorlinks,linkcolor=red,citecolor=blue,urlcolor=blue,hypertexnames=false]{hyperref}
\hypersetup{
unicode=true
}
\usepackage{graphics}
\usepackage{tensor}
\usepackage{xcolor}
\usepackage{mathcommand}
\usepackage{todonotes}
\numberwithin{equation}{section}
\newtheorem{thm}[equation]{Theorem}
\newtheorem{prop}[equation]{Proposition}
\newtheorem{coro}[equation]{Corollary}
\newtheorem{lem}[equation]{Lemma}

\theoremstyle{definition}
\newtheorem{defi}[equation]{Definition}
\newtheorem{rem}[equation]{Remark}
\newtheorem{exa}[equation]{Example}
\usepackage{todonotes}

\newcommand{\Alg}{{\mathbf{Alg}}}
\newcommand{\Top}{{\mathbf{Top}}}
\newcommand{\sSet}{\mathbf{sSet}}
\newcommand{\Map}{\mathrm{Map}}
\newcommand{\uhom}{{\underline \hom}}
\newcommand{\uHOM}{{\underline{\mathrm{HOM}}}}
\newcommand{\diag}{\mathrm{diag}}
\newcommand{\id}{\mathrm{id}}
\newcommand{\Exinf}{{\mathrm{Ex}^{\infty}}}
\newcommand{\colim}{\operatornamewithlimits{colim}}
\newcommand{\sd}{\mathrm{sd}}
\newcommand{\op}{\mathrm{op}}
\newcommand{\ob}{\mathrm{ob}}
\newcommand{\const}{\mathrm{const}}

\newcommand{\bigast}{\mathop{\scalebox{1.5}{\raisebox{-0.2ex}{$\ast$}}}}
\newcommand{\Fib}{\mathit{Fib}}
\newcommand{\ev}{\mathrm{ev}}

\newcommand{\Z}{\mathbb{Z}}
\newcommand{\N}{\mathbb{N}}
\newcommand{\Minf}{M_\infty}
\newcommand{\kk}{{kk}}
\newcommand{\Fun}{\mathrm{Fun}}
\newcommand{\ho}{\mathrm{Ho}}
\newcommand{\ind}{\mathrm{ind}}
\newcommand{\concat}{\mathrm{concat}}
\newcommand{\clas}{\mathrm{clas}}
\newcommand{\fk}{\mathfrak{K}}
\newcommand{\ks}{{\fk_s}}
\newcommand{\kf}{{\fk_f}}
\newcommand{\cC}{\mathcal{C}}
\newcommand{\cD}{\mathcal{D}}

\newcommand{\cF}{\mathcal{F}}

\newcommand{\ufk}{\underline{\fk}}
\newcommand{\uks}{{\ufk_s}}
\newcommand{\ukf}{{\ufk_f}}
\newcommand{\tF}{\tilde{F}}
\newcommand{\hF}{\hat{F}}
\newcommand{\vF}{\check{F}}

\newcommand{\scrC}{\mathscr{C}}
\newcommand{\scrU}{\mathscr{U}}
\newcommand{\scrP}{\mathscr{P}}
\newcommand{\scrE}{\mathscr{E}}
\newcommand{\inc}{\mathrm{incl}}
\newcommand{\pr}{\mathrm{pr}}
\newcommand{\Iso}{\mathrm{Iso}}

\newcommand{\MOm}{{\mathcal{M}^\Omega}}
\newcommand{\CAlg}{{C^*\Alg}}
\newcommand{\FSch}{{\cF_{Sch}}}
\newcommand{\KK}{\mathrm{KK}}
\newcommand{\scrT}{\mathscr{T}}
\newcommand{\KH}{\operatorname{KH}}

\newcommand{\bbC}{\mathbb{C}}
\newcommand{\cM}{\mathcal{M}}

\begin{document}

\title{Homotopy structures realizing algebraic kk-theory}
\author{Eugenia Ellis}
\email{eellis@fing.edu.uy}
\address{IMERL. Facultad de Ingenier\'\i a. Universidad de la Rep\'ublica. Montevideo, Uruguay.}
\author{Emanuel Rodr\'iguez Cirone}
\email{erodriguezcirone@cbc.uba.ar}
\address{Dep. Matemática -- CBC -- UBA, Buenos Aires, Argentina.}

\thanks{All authors were partially supported by grants PICT-2021-I-A-00710, \emph{$K$-theory, homology, and noncommutative geometry}, and UBACyT 2023 20020220300206BA, \emph{Álgebra, geometría y topología no conmutativas}. The first author was partially supported by ANII, CSIC and PEDECIBA. The second author was partially supported by grant PIP 2021-2023 GI, 11220200100423CO, \emph{Álgebra y geometría no conmutativas}.}


\begin{abstract}
Algebraic $\kk$-theory, introduced by Corti\~nas and Thom, is a bivariant $K$-theory defined on the category $\Alg$ of algebras over a commutative unital ring $\ell$. 
It consists of a triangulated category $\kk$ endowed with a functor from $\Alg$ to $\kk$ that is the universal excisive, homotopy invariant and matrix-stable homology theory.
Moreover, one can recover Weibel's homotopy $K$-theory $\KH$ from $\kk$ since we have $\kk(\ell,A)=\KH(A)$ for any algebra $A$. 
We prove that $\Alg$ with the split surjections as fibrations and the $\kk$-equivalences as weak equivalences is a stable category of fibrant objects, whose homotopy category is $\kk$.
As a consecuence of this, we prove that the Dwyer-Kan localization $\kk_\infty$ of the $\infty$-category of algebras at the set of $\kk$-equivalences is a stable infinity category whose homotopy category is $\kk$.
\end{abstract}

\maketitle

\section{Introduction}
Kasparov's $KK$-theory, introduced in \cite{kasparov}, is the major tool in noncommutative topology. To every pair $(A, B)$ of separable $C^*$-algebras it associates an abelian group $KK(A, B)$ that generalizes both of topological $K$-theory and topological $K$-homology. These groups were deeply studied by Cuntz, who gave in \cite{newlook} an alternative description of them and provided a new perspective on the theory. The groups $KK(A, B)$ are the hom-sets of an additive category $KK$ whose objects are separable $C^*$-algebras. Higson proved in \cite{higson} that $KK$ is the target of the universal homotopy invariant, $C^*$-stable and split-exact functor from the category $\CAlg$ of separable $C^*$-algebras into an additive category. Meyer and Nest proved in \cite{mn} that the category $KK$ is actually triangulated in a natural way. Cuntz also analyzed $KK$-theory from an algebraic standpoint and defined in \cite{cuntzlc} a bivariant $K$-theory for all locally convex algebras.

Motivated by the works of Cuntz and Higson, algebraic $\kk$-theory was introduced by Corti\~nas and Thom in \cite{cortho} as a completely algebraic counterpart of Kasparov's $KK$-theory. To every pair $(A,B)$ of algebras over a commutative ring $\ell$ it associates an abelian group $\kk(A,B)$ that generalizes Weibel’s homotopy $K$-theory $\KH$, defined in \cite{chuck}. The groups $kk(A,B)$ are the hom-sets of a triangulated category $kk$ whose objects are $\ell$-algebras. This category $\kk$ is the target of the universal polynomial homotopy invariant, $\Minf$-stable and excisive functor from the category $\Alg$ of $\ell$-algebras into a triangulated category.

Triangulated categories often appear (but not always; see \cite{muro}) as the homotopy categories of stable $\infty$-categories. A natural question is whether there exist stable $\infty$-categories whose homotopy categories are the bivariant $K$-theory categories mentioned above. For Kasparov's $KK$-theory this is answered affirmatively by Land and Nikolaus in \cite{ln}. Indeed, they construct a stable $\infty$-category $KK_\infty$ whose homotopy category is $KK$, upon performing the Dwyer-Kan localization of the $\infty$-category $\CAlg$ at the set of $KK$-equivalences. The same methods are used in \cite{bel} to establish the $G$-equivariant case for a countable group $G$. A different approach is taken in \cite{bunke}, where a stable $\infty$-category realizing $KK$-theory is constructed independently of the classical $KK$-groups. In this paper we  address this question for algebraic $\kk$-theory---the $G$-equivariant case \cite{ekk} will be addressed in a separate work. Our main result is the following:
\begin{thm}[Theorem \ref{thm:main}]\label{thm:intromain}
    Let $\Alg$ be the category of $\ell$-algebras and let $W_\kk$ be the set of $\kk$-equivalences in $\Alg$. Let $\kk_\infty:=\Alg[W_\kk^{-1}]$ be the Dwyer-Kan localization. Then $\kk_\infty$ is a stable $\infty$-category whose homotopy category is triangle equivalent to $\kk$.
\end{thm}
To prove this result we follow the path taken in \cite{ln}. A key tool for controlling finite limits in $\kk_\infty$ is the fact that $\Alg$ admits the structure of a category of fibrant objects where the weak equivalences are the $\kk$-equivalences. This fact is proved in Proposition \ref{prop:algcof} and can be considered as an analogue of \cite{uuye}*{Theorem 2.29} in the algebraic context. To show that the loop functor in $\kk_\infty$ is an equivalence, we use that the homotopy category of this category of fibrant objects is the category $\kk$ defined by Corti\~nas-Thom; the latter is established in Theorem \ref{kkuniprop}. Along the way, we prove the following theorem, which is of independent interest.
\begin{thm}[Theorem \ref{thm:kkloc}]
    Let $W_H$ be the family of polynomial homotopy equivalences in $\Alg$. Let $W_S=\{A\to \Minf A\}_{A\in\Alg}$ be the family of upper-left corner inclusions. Let $W_E$ be the family of classifying maps of those extensions whose middle term is either contractible or an infinite-sum ring.
    Then the functor $j:\Alg\to\kk$ is initial in the category of those functors $F:\Alg\to \cC$ such that:
    \begin{itemize}
        \item $\cC$ is a category with finite products and $F$ preserves finite products;
        \item $F$ sends morphisms in $W_H\cup W_S\cup W_E$ to isomorphisms.
    \end{itemize}
\end{thm}

The stability of $\kk_\infty$ implies that $\kk$-theory is naturally enriched over spectra. A point-set level construction of a spectrum representing $\kk$-theory was presented in \cite{gar1}. This was later used in \cite{gar2} to prove that $\kk$ is contravariantly equivalent to a full subcategory of certain motivic stable homotopy category. Our methods are completely different from those used in \cite{gar1} and \cite{gar2}.

In \cite{bunke}, a stable $\infty$-category representing Kasparov's $KK$-theory is constructed by enforcing the universal properties one by one through a sequence of localizations. It is suggested in loc.\ cit.\ that this method, with some modifications, should also work for algebraic $\kk$-theory. The first localization performed in \cite{bunke} is the Dwyer-Kan localization of $\CAlg$ at the homotopy equivalences; write $\CAlg_h$ for the resulting $\infty$-category.
It is proved in \cite{bunke}*{Proposition 3.5} that $\CAlg_h$ is equivalent to coherent nerve of $\CAlg$, considered as a Kan-complex enriched category. In the algebraic context, we get the following result.
\begin{prop}
    Let $\Alg_h$ be the Dwyer-Kan localization of $\Alg$ at the polynomial homotopy equivalences. Then $\Alg_h$ is equivalent to the coherent nerve of $\Alg$, considered as a Kan-complex enriched category with the hom-spaces $\uHOM(A,B)$ defined in \eqref{homkan}.
\end{prop}
Back on the topological side, it is proved in \cite{bunke}*{Proposition 3.17} that $\CAlg_h$ is left-exact using that $\CAlg$ admits the structure of a category of fibrant objects with homotopy equivalences as weak equivalences \cite{uuye}*{Thm. 2.19}. At this point, an important difference arises between the topological and algebraic contexts. We show in Proposition \ref{prop:algfnotcof} that there is no reasonable family of fibrations making $\Alg$ into a category of fibrant objects with polynomial homotopy equivalences as weak equivalences. As a consecuence, we do not know whether the $\infty$-category $\Alg_h$ has pullbacks. This indicates that polynomial homotopies behave worse than continuous ones and suggests that more changes are needed to translate the methods of \cite{bunke} to the algebraic context.

\section{Preliminaries}
We recall definitions and results that will be used later on. Throughout this paper, $\ell$ is a commutative ring with unit and $\Alg$ is the category of associative not necessarily unital $\ell$-algebras.

\subsection{Homotopies}\label{aho}

\subsubsection{Homotopy of continuous maps}
Let $f,g:X\to Y$ be continuous maps of topological spaces. We say that $f$ is \emph{homotopic} to $g$, denoted by $f\sim g$, if there exists a continuous map $H:X\times [0,1]\to Y$ such that $H(x,0)=f(x)$ and $H(x,1)=g(x)$. Note that $H$ can be viewed as a continuous map $H:X\to C([0,1],Y)$ such that $\ev_0\circ H = f$ and $\ev_1\circ H = g$.
Here, $C([0,1],Y)$ denotes the space of continuous maps $[0,1]\to Y$.
It is well known that $\sim$ is an equivalence relation. Indeed, to show transitivity we can proceed as follows. Let $H$ be a homotopy from $f$ to $g$ and $H'$ be a homotopy from $g$ to $h$. By the universal property of the pullback, we have a unique map $H''$ making the following diagram commute.
\begin{equation}\label{eq:hotop}\begin{gathered}\xymatrix@C=1.5em@R=3em{
X \ar@/^1pc/[drr]^{H'}\ar@![dr]^{\exists ! \;H''}\ar@/_2pc/[ddr]_-{H} &&\\ 
& C([0,1],Y)\underset{Y}{\times}C([0,1],Y)\ar[d]_-{\pr_1}\ar[r]^-{\pr_{2}} & C([0,1],Y)
\ar[d]^{\ev_0}\\
& C([0,1],Y)\ar[r]_-{\ev_1}& Y \ar@{}[ul]|(0.7){\text{\scalebox{2}{$\lrcorner$}}}
}\end{gathered}\end{equation}
Since $[0,1]\lor [0,1]\cong [0,\tfrac12]\lor [\tfrac12, 1]=[0,1]$, we have $C([0,1],Y)\times_YC([0,1],Y)\cong C([0,1],Y)$ and we can consider $H''$ as a continuous map $X\to C([0,1],Y)$. This $H''$ is a homotopy from $f$ to $h$.

\subsubsection{Homotopy of algebra homomorphisms}\label{sec:hah}
Let $f,g: A\to B$ be algebra homomorphisms. We say that $f$ is \emph{elementary homotopic} to $g$, denoted by $f\sim_e g$, if there exists an algebra homomorphism $H:A\to B[t]$ such that
$\ev_0\circ H= f$ and $\ev_1\circ H=g$, where $\ev_i:B[t]\to B$ is the evaluation at $i$ for $i=0,1$.
Note that $\sim_e$ is reflexive and symmetric, but not transitive. Indeed, since $B[t]\times_B B[t]\not\cong B[t]$, we cannot obtain a new homotopy by glueing two homotopies.
We say that $f$ is {\emph{homotopic}} to $g$, denoted by $f\sim g$, if there exist $h_1$,\ldots, $h_n\in \hom_{\Alg}(A,B)$ such that
\[
f\sim_e h_1 \sim_e \cdots \sim_e h_n \sim_e g.
\]
Put $B_n[t]=\{(p_0,\ldots,p_n)\in B[t]\times \cdots \times B[t]\mid \ev_1(p_i)=\ev_0(p_{i+1})\text{ for $0\leq i <n$} \}$. 
Define $\ev_i:B_n[t]\to B$ by $\ev_0(p_0,\ldots,p_n)= \ev_0(p_0)$ and $\ev_1(p_0,\ldots,p_n)= \ev_1(p_n)$. 
We have a directed system
\[
B_{\bullet}[t]: B[t] \xrightarrow{\beta_0} B_1[t] \xrightarrow{\beta_1} \cdots \to B_n[t] \xrightarrow{\beta_n} B_{n+1}[t]\xrightarrow{\beta_{n+1}}\cdots
\]
where $\beta_n(p_0,\ldots,p_n)= (p_0,\ldots,p_n, \ev_1(p_n))$. Note that $\ev_i$ induces a morphism $\ev_i:B_\bullet[t]\to B$ in the category of ind-algebras. With this notation, $f$ is homotopic to $g$ if there exists a morphism of ind-algebras $H:A\to B_\bullet[t]$ such that $\ev_0\circ H= f$ and $\ev_1\circ H=g$. Homotopy is a transitive relation. 
Indeed, let $f,h,g \in \hom_{\Alg}(A,B)$ such that $f\sim h$ and $h\sim g$. Then there exist algebra homomorphisms $H:A\to B_n[t]$ and $H':A\to B_{m}[t]$ such that
$\ev_0\circ H = f$, $\ev_1\circ H= h$, $\ev_0\circ H'= h$ and $\ev_1 \circ H'= g$. 
By the universal property of the pullback, there is a unique morphism $H'': A \to B_n[t]\times_BB_{m}[t]= B_{n+m}[t]$ such that $\ev_0\ \circ H'' = f$ and $\ev_1\ \circ H''= g$.
\[
\xymatrix@C=1.5em@R=3em{
A \ar@/^1pc/[drr]^{H'}\ar@![dr]^{\exists ! \;H''}\ar@/_1pc/[ddr]_{H} &&\\ 
& B_n[t]\times_BB_{m}[t]\ar[d]_-{pr}\ar[r]^-{pr} &B_{m}[t]
\ar[d]^{\ev_0}\\
& B_n[t]\ar[r]_{\ev_1}&B\ar@{}[ul]|(.70){\text{\scalebox{2}{$\lrcorner$}}}
}
\]
Homotopy is compatible with composition of algebra homomorphisms. We have a category $[\Alg]$ whose objects are the $\ell$-algebras and whose hom-sets are defined by:
\[\hom_{[\Alg]}(A,B)=[A,B]:=\hom_\Alg(A,B)/\sim\]

\subsection{Simplicial enrichment of algebras}\label{sec:enrichKan}
Let us recall the simplicial enrichment in $\Alg$ introduced in \cite{cortho}. 
Let $B$ be an algebra. For $n\geq 0$, the algebra $B^{\Delta^n}$ of \emph{$B$-valued polynomial functions on the standard $n$-simplex} is defined as $B^{\Delta^n}:=B[t_0,\dots, t_n]/\langle t_0+\cdots + t_n-1\rangle$. For a simplicial set $X$, the algebra of \emph{$B$-valued polynomial functions on $X$} is defined as $B^X:=\hom_\sSet(X,B^\Delta)$
where $B^\Delta$ is the simplicial algebra $[n]\mapsto B^{\Delta^n}$.

The category $\Alg$ can be enriched over simplicial sets, as we proceed to recall. For a pair of algebras $(A,B)$, let $\uhom(A,B)$ be the simplicial set defined by:
\[[n]\mapsto \hom_{\Alg}(A, B^{\Delta^n})\]
We have a simplicial composition
\[\circ:\uhom(B,C)\times\uhom(A,B)\to\uhom(A,C)\]
defined as follows. Let $f\in\uhom(A,B)_n$ and $g\in\uhom(B,C)_n$ be represented by $a:A\to B^{\Delta^n}$ and $b:B\to C^{\Delta^n}$ respectively. Then $g\circ f$ is represented by the composite
\begin{equation}\label{eq:simplicialcomposition}\xymatrix{A\ar[r]^-{a} & B^{\Delta^n}\ar[r]^-{b^{\Delta^n}} & (C^{\Delta^n})^{\Delta^n}\ar[r]^-{\mu} & C^{\Delta^n\times\Delta^n}\ar[r]^-{\diag^*} & C^{\Delta^n} }\end{equation}
where $\mu$ is the morphism defined in \cite{htpysimp}*{Section 3.1}.

\begin{rem}
    Upon identifying $C^{\Delta^n}\cong C\otimes \ell^{\Delta^n}$, the composite $\diag^*\circ \mu$ in \eqref{eq:simplicialcomposition} equals the morphism
$\id_C\otimes m_{\Delta^n}:C\otimes\ell^{\Delta^n}\otimes\ell^{\Delta^n}\to C\otimes\ell^{\Delta^n}$
where $m_{\Delta^n}$ is the multiplication in the commutative algebra $\ell^{\Delta^n}$.
\end{rem}

Upon applying $\Exinf$ to the enrichment described above we get an enrichment of $\Alg$ over Kan complexes, see \cite{cortho}*{Thm. 3.2.3}. Explicitely, we have:
\begin{equation}\label{homkan}
\uHOM(A,B):=\Exinf\uhom(A,B)= \colim_r\hom_{\Alg}(A, B^{\Delta^\bullet}_r)
\end{equation}
Here we write $B^{\Delta^n}_r$ for $B^{\sd^r\Delta^n}$, where $\sd^r\Delta^n$ is the $r$-fold subdivision of $\Delta^n$.

\begin{rem}
    For $B\in\Alg$, $X\in\sSet$ and $r\geq 0$, let $B^X_r$ denote $B^{\sd^rX}$, where $\sd^r\Delta^n$ is the $r$-fold subdivision of $\Delta^n$. As explained in \cite{cortho}*{Section 3.2}, we have an ind-algebra $B^X_\bullet$ with transition morphisms induced by the last vertex map.
\end{rem}
\begin{rem} Two algebra homomorphisms $f,g:A\to B$ are homotopic if and only if there exist $r\in \N$ and $H: A \to B_{r}^{\Delta^1}$ such that $d_0\circ H =f$ and $d_1\circ H=g$.
\end{rem}

\subsection{Categories of fibrant objects}\label{lcof}
Following \cite{brown}, a {\emph{category of fibrant objects}} is a category $\cC$ with terminal object $*$ and distinguished subcategories  $\cF$ and $W$ such that:
\begin{enumerate}
\item[$(F_1)$] The isomorphisms of $\cC$ are in $\cF$.
\item[$(F_2)$] The pullback in $\cC$ of a morphism in $\cF$ exists and is in $\cF$.
\item[$(F_3)$] For any object $B$ of $\cC$, the morphism $B\rightarrow *$ is in $\cF$.
\item[$(W_1)$] The isomorphisms of $\cC$ are in $W$.
\item[$(W_2)$] If two of $f$, $g$ and $gf$ are in $W$, then so is the third.
\item[$(FW_1)$] The pullback in $\cC$ of a morphism in $W\cap \cF$ is in $W\cap \cF$.
\item[$(FW_2)$] For any object $B$ of $\cC$, the diagonal map $B\rightarrow B\times B$ admits a factorization 
\[
\xymatrix{B \ar[r]^-{\sim}_-{s}& B^I \ar@{->>}[r]_-{d} &B\times B }
\]
where $s\in W$ and $d\in\cF$.
The triple $(B^I, s, d)$ is called a $\emph{path-object}$ for $B$.
\end{enumerate}
The morphisms in $\cF$ are called {\emph{fibrations}} and are denoted by $\twoheadrightarrow$.
The morphisms in $W$ are called {\emph{weak equivalences}} and are denoted by $\overset{\sim}{\to}$.
The morphisms in $W\cap \cF$ are called {\emph{trivial fibrations}} and are denoted by $\overset{\sim}{\twoheadrightarrow}$.
    
The structure of a category of fibrant objects $\cC$ with weak equivalences $W$ is a tool to deal with the localization $\cC[W^{-1}]$. The latter is called the \emph{homotopy category} of $\cC$ and is denoted by $\ho(\cC)$.

\begin{lem}\label{lem:cofproducts}
    Let $(\cC, \cF, W)$ be a category of fibrant objects. Then $\cC[W^{-1}]$ has finite products and the localization functor $\cC\to \cC[W^{-1}]$ commutes with finite products.
\end{lem}
\begin{proof}
It suffices to show that $f\times g\in W$ whenever $f,g\in W$. Let $f:A\to A'$ and $g:B\to B'$ morphisms in $W$ and consider the following pullbacks of algebras:
\[\xymatrix{A\times B\ar[d]_-{f\times \id}\ar[r] & A\ar[d]_-{\sim}^-f \\
A'\times B\ar@{->>}[r] & A'}\qquad \qquad
\xymatrix{A'\times B\ar[d]_-{\id\times g}\ar[r] & B\ar[d]_-{\sim}^-{g} \\
A'\times B'\ar@{->>}[r] & B'}\]
By \cite{brown}*{I.4 Lemma 2}, we have that $f\times \id$ and $g\times\id$ are weak equivalences and thus $f\times g=(\id\times g)\circ(f\times \id)$ is a weak equivalence as well. 
\end{proof}

\subsubsection{Examples} For any model category $\cM$, let $\cM_f$ be the full subcategory consisting of the fibrant objects. Then $\cM_f$ is naturally a category of fibrant objects, with weak equivalences and fibrations restricted from $\cM$. 
For example, let $\Top$ be the category of compactly generated weakly Hausdorff topological spaces together with continuous maps.
Then $\Top$ is a category of fibrant objects with $W$ the weak homotopy equivalences and $\cF$ the Serre fibrations.

Two examples in the context of $C^*$-algebras are provided in \cite{uuye}. Let $\CAlg$ be the category of separable $C^*$-algebras with $*$-homomorphisms. The category $\CAlg$ is naturally enriched over $\Top$ upon endowing $\hom_\CAlg(A,B)$ with the compact-open topology; see \cite{uuye}*{Remark 2.1}. 
If $B$ is a $C^*$-algebra and $X$ is a compact Hausdorff space, we write $B^{X}$ for the $C^*$-algebra $\hom_{\Top}(X,B)\cong C(X)\otimes B$. Let $\ev_{x}:B^{X} \rightarrow B$ be the evaluation at $x\in X$.
Two $*$-homomorphisms $f_0, f_1: A\rightarrow B$ are \emph{homotopic} if there exists a $*$-homomorphism $H:A\rightarrow B^I$, with $I=[0,1]$, such that $f_0=\ev_0 \circ H$ and $f_1=\ev_1\circ H$. In this case, we write $f_0\sim f_1$. Denote the set of homotopy classes of $*$-homomorphisms $A\to B$ by:
\[
[A,B]=\hom_{\CAlg}(A,B)/\sim
\]
We have a category $[\CAlg]$ whose objects are the separable $C^*$-algebras and whose hom-sets are the sets $[A,B]$.
A $*$-homomorphism $f$ is a \emph{homotopy equivalence} if $[f]$ is invertible in $[\CAlg]$. By
\cite{uuye}*{Prop. 2.13}, $f$ is a homotopy equivalence if and only if the induced map $f_*:\hom_\CAlg(D,A)\rightarrow  \hom_\CAlg(D,B)$
is a weak equivalence in $\Top$ for all $D\in\CAlg$.
A $*$-homomorphism $p:E\rightarrow B$ is called a {\emph{Schochet fibration}} \cite{scho} if the induced map $p_*:\hom_\CAlg(D,E)\rightarrow  \hom_\CAlg(D,B)$
is a Serre fibration in $\Top$ for all $D\in\CAlg$. 
By \cite{uuye}*{Theorem 2.19}, we have a category of fibrant objects $(\CAlg, \FSch, W_{h})$ where the weak equivalences are the homotopy equivalences and the fibrations are the Schochet fibrations. The corresponding homotopy category is $[\CAlg]$. 
This example is not of the form $\cM_f$ for any model category $\cM$; see \cite{uuye}*{Appendix}. 

Let $KK$ be the category whose objects are the separable $C^*$-algebras and whose hom-sets are the Kasparov groups $KK(A,B)$. A $*$-homomorphism $f:A \to B$ is a \emph{$KK$-equivalence} if the induced morphism $f_*: KK(D,A)\to KK(D,B)$ is an isomorphism for all $D \in \CAlg$.
By \cite{uuye}*{Theorem 2.29}, we have a category of fibrant objects $(\CAlg, \FSch, W_{KK})$ where the weak equivalences are the $KK$-equivalences and the fibrations are the Schochet fibrations.

\subsection{Bivariant algebraic \texorpdfstring{$K$-theory}{K-theory}}\label{sec:kk}
\label{dkk}
Let us recall the details of the algebraic $\kk$-theory introduced in \cite{cortho}.  An alternative construction of $\kk$-theory was given in \cite{garkuni}. An \emph{extension} is a short exact sequence  of $\ell$-algebras
\begin{equation}\label{eq:ext}\scrE:\xymatrix{A\ar[r] & B\ar[r] & C}\end{equation}
that splits as a sequence of $\ell$-modules. An \emph{excisive homology theory} on $\Alg$ consists of a functor $H:\Alg\to\scrT$ into a triangulated category $\scrT$ together with a morphism $\partial_{\scrE}:\Omega H(C)\to H(A)$ for every extension $\scrE$ such that:
\begin{enumerate}
    \item the triangle $\Omega H(C)\overset{\partial_\scrE}{\to} H(A)\to H(B)\to H(C)$ is distinguished for every extension $\scrE$;
    \item the morphisms $\partial_\scrE$ are natural with respect to morphisms of extensions---a \emph{morphism of extensions} is just a morphism of short exact sequences in $\Alg$.
\end{enumerate}
Here we write $\Omega$ for the desuspension functor in $\scrT$. We say that $H$ is \emph{homotopy invariant} if it sends polynomially homotopic morphisms to the same morphism. We say that $H$ is \emph{$\Minf$-stable} if it sends the upper-left corner inclusion $A\to \Minf A$, $a\mapsto a\cdot e_{1,1}$, to an isomorphism for every algebra $A$. Here, we write $\Minf A$ for the algebra of finitely supported matrices with coefficients in $A$ indexed over $(\Z_{\geq 1})^2$.
\begin{thm}[\cite{cortho}*{Theorem 6.6.2}]
    There exists a triangulated category $\kk$ endowed with a functor $j:\Alg\to \kk$ that is the universal homotopy invariant and $\Minf$-stable excisive homology theory. That is, any homotopy invariant and $\Minf$-stable excisive homology theory factors uniquely through $j$.
\end{thm}
The triangulated category $\kk$ and the functor $j$ are uniquely determined up to equivalence of triangulated categories. The category $\kk$ can be constructed so that its objects are the $\ell$-algebras and the functor $j$ is the identity on objects. We often write $\kk(A,B)$ instead of $\kk(j(A), j(B))$.

For $A,B\in\Alg$ and $n\in\Z$, put $\kk_n(A,B):=\kk(A,\Omega^nB)$. These groups fit into a long exact sequence
\[\xymatrix@C=1.5em{\cdots\ar[r] & \kk_{n+1}(D, C)\ar[r] & \kk_n(D,A)\ar[r] & \kk_n(D,B)\ar[r] & \kk_n(D, C)\ar[r] & \cdots}\]
upon applying $\kk(D,-)$ to the extension \eqref{eq:ext}.
\begin{thm}[\cite{cortho}*{Theorem 8.2.1}] Let $A$ be an algebra and let $\KH_*(A)$ be Weibel's homotopy $K$-theory groups of $A$ \cite{chuck}. Then there exists a natural isomorphism
\[
\kk_*(\ell,A)\cong \KH_*(A).
\]
\end{thm}
 
For $A\in\Alg$, put $PA:=tA[t]$ and $\Omega A:=(t^2-t)A[t]$. These are called the \emph{path algebra} and the \emph{loop algebra} of $A$ respectively and fit into the \emph{loop extension} of $A$; see \cite{cortho}*{Sec. 4.5}. The desuspension $\Omega$ in the triangulated category $\kk$ is induced by these loop algebras; see \cite{cortho}*{Lemma 6.3.9}. Let $f:A\to B$ be an algebra homomorphism. Let $P_f$ be defined by the following morphism of extensions where the square on the right is a pullback:
\begin{equation}\label{eq:pathseq}\begin{gathered}\xymatrix{
\Omega B\ar@{=}[d]\ar[r]^-{\iota_f} & P_f\ar[d]\ar[r]^-{\pi_f} & A\ar[d]^-{f} \\
\Omega B\ar[r]^-{\inc} & PB\ar[r]^-{\ev_1} & B
}\end{gathered}\end{equation}
By \cite{cortho}*{Def. 6.5.1}, a triangle in $\kk$ is {\emph{distinguished}} if and only if it is isomorphic to one of the form
\[\xymatrix{\Omega B\ar[r]^-{j(\iota_f)} & P_f\ar[r]^-{j(\pi_f)} & A\ar[r]^-{j(f)} & B}\]
for some algebra homomorphism $f:A\to B$.

\section{Category of fibrant objects structures on \texorpdfstring{$\Alg$}{Alg}}\label{seccof}
Let $W_{H}$ be the class of polynomial homotopy equivalences in $\Alg$, and let $W_{kk}$ be the class of $kk$-equivalences---that is, those algebra homomorphisms that become isomorphisms in $\kk$. In this section we study the existence of category of fibrant objects structures on $\Alg$ having each of these classes as the class of weak equivalences. Consider the class $\Fib$ of surjective morphisms in $\Alg$ that have a linear section. We prove in Proposition \ref{prop:algcof} that $(\Alg, \Fib, W_{kk})$ is a category of fibrant objects. This result can be considered as an algebraic analogue of \cite{uuye}*{Theorem 2.29}. Moreover, in Proposition \ref{prop:algfnotcof} we prove that there is no reasonable class of fibrations $\cF$ making $(\Alg, \cF, W_H)$ into a category of fibrant objects. This shows that the algebraic counterpart of \cite{uuye}*{Theorem 2.19} does not hold. 
This happens because polynomial homotopies behave worse than continuous ones, as it is shown in the following lemma.

\begin{lem}\label{betano}
Let $\ell$ be a domain such that $\KH_0(\ell)\neq 0$---this happens, for example, if $\ell$ is a field or a principal ideal domain. 
Put $\Omega_0 = (t^2-t)\ell[t]$ and $\Omega_1=\{(p,q)\in\ell[t]\times\ell[t]: \ev_1(p)=\ev_0(q),\  \ev_0(p)=\ev_1(q)=0\}$. 
Then the morphism $\beta: \Omega_0 \to \Omega_1$ defined by $\beta(p)=(p, 0)$ is not a polynomial homotopy equivalence.
\end{lem}
\begin{proof}
 Suppose that $\beta$ is a polynomial homotopy equivalence. Then $[\beta]$ is an isomorphism in $[\Alg]$ and it has an inverse $[\alpha]$ where $\alpha$ is an algebra homomorphism $\alpha:\Omega_1 \rightarrow \Omega_0$.  
 In particular, $[\alpha]\circ[\beta]=[\id_{\Omega_0}]$. We have:
\[\alpha(t^2-t,0)\alpha(0,t^2-t)=\alpha\left((t^2-t,0)(0,t^2-t)\right)=\alpha(0,0)=0 \in \Omega_0\subset \ell[t]
\]
As $\ell[t]$ is a domain, we should have $\alpha(t^2-t,0)=0$ or $\alpha(0,t^2-t)=0$. Suppose that $\alpha(t^2-t,0)=0$. Then we have
\begin{align*}
    \alpha\left((t^2-t)\left(\textstyle\sum\limits_{0}^na_it^i\right), 0\right)&=\alpha\left((a_0(t^2-t), 0)+\left((t^2-t)\left(\textstyle\sum\limits_{1}^na_it^i\right),0\right)\right) \\
    &=a_0\alpha\left(t^2-t,0\right)+\alpha(t^2-t,0)\alpha\left(\textstyle\sum\limits_{1}^na_it^i,\textstyle\sum\limits_{1}^na_i(1-t)^i\right)\\
    &=0
\end{align*}
showing that $\alpha\circ \beta$ is the zero morphism. Then $[\id_{\Omega_0}]=[0]$ and $\Omega_0\cong 0$ in $[\Alg]$. 
Since $j:\Alg \to \kk$ factors through $[\Alg]$, then $\Omega_0\cong 0$ in $\kk$. But this cannot happen since we have:
\[0=\kk(\Omega_0, \Omega_0)=\kk(\Omega \ell, \Omega\ell)\cong\kk(\ell, \ell)\cong KH_0(\ell)\neq 0\]
If $\alpha(0,t^2-t)=0$, we can show that $\alpha\circ\tilde{\beta}$ is the zero morphism, where $\tilde{\beta}:\Omega_0\to\Omega_1$ is given by $\tilde{\beta}(p)=(0,p)$. But $\beta$ and $\tilde{\beta}$ are homotopic morphisms; indeed, an elementary homotopy $H:\Omega_0\to \Omega_1[s]$ is given by:
\[H(p)=(p(t(1-s)), p(1-s(1-t)))\]
Then $[\id_{\Omega_0}]=[\alpha]\circ[\beta]=[\alpha]\circ[\tilde{\beta}]=[\alpha\circ\tilde{\beta}]=0$ and we get to the same contradiction as before. This finishes the proof.
\end{proof}

\begin{rem}
    It is easily verified using excision that $\beta$ is a $\kk$-equivalence; see \cite{cortho}*{Sec. 6.3}.
\end{rem}

\begin{rem}
    The analogue of $\beta$ in the topological context is identified with the morphism $\beta:\bbC(0,1)\to \bbC(0,1)$ given by
    \[\beta(f)(t)=\begin{cases}
        f(2t) & \text{if $0\leq t\leq \tfrac12$,}\\
        0 & \text{if $\tfrac12\leq t\leq 1$.}
    \end{cases}\]
    Here, $\bbC(0,1)$ denotes the $C^*$-algebra of continuous functions $[0,1]\to\bbC$ that vanish at $t=0,1$. This $\beta$ is easily seen to be a homotopy equivalence.
\end{rem}

Recall the definition of the Kan complex $\uHOM(A,B)$ from \eqref{homkan}. The following propositions will be useful later on.

\begin{prop}[cf. \cite{uuye}*{Prop. 2.13}]\label{indhom}
    Let $f:A\to B$ be an algebra homomorphism. Then $f$ is a homotopy equivalence if and only if for every algebra $D$, the induced morphism
    \begin{equation}\label{eq:inducedHOM}f_*:\uHOM(D,A)\to \uHOM(D,B)\end{equation}
    is a weak equivalence of simplicial sets.
\end{prop}
\begin{proof}
    By \cite{htpysimp}*{Theorem 3.10} for every $n\geq 0$ we have a natural bijection
    \[\pi_n\uHOM(D,A)\cong [D, A^{S^n}_\bullet]=\colim_r[D, A^{S^n}_r]\]
    where $A^{S^n}_\bullet$ is the ind-algebra of polynomials on the $n$-dimensional cube with coefficients in $A$ vanishing at the boundary of the cube; see \cite{htpysimp}*{Ex. 2.11}.

    Suppose that \eqref{eq:inducedHOM} is a weak equivalence of simplicial sets for every algebra $D$. Upon applying $\pi_0$ we get bijections
    \[f_*:[D,A]\to [D,B]\]
    for every algebra $D$. By Yoneda, the latter implies that $f$ is an isomorphism in $[\Alg]$ --- i.e.\ a polynomial homotopy equivalence.

    Now suppose that $f$ is a polynomial homotopy equivalence and fix an algebra $D$. To prove that \eqref{eq:inducedHOM} is a weak equivalence of simplicial sets, we must show that it induces a bijection upon applying $\pi_n$, for all $n\geq 0$. By \cite{htpysimp}*{Lem. 2.10}, we have algebra isomorphisms $A^{S^n}_r\cong A\otimes\Z^{S^n}_r$ for all $n,r\geq 0$. Since tensoring with a ring preserves polynomial homotopy equivalences, we have a morphism of diagrams:
    \[\xymatrix{
    [D,A^{S^n}_0]\ar[r]\ar[d]^-{\cong}_-{f_*} & [D,A^{S^n}_1]\ar[r]\ar[d]^-{\cong}_-{f_*} & [D,A^{S^n}_2]\ar[r]\ar[d]^-{\cong}_-{f_*} & \cdots \\
    [D,B^{S^n}_0]\ar[r] & [D,B^{S^n}_1]\ar[r] & [D,B^{S^n}_2]\ar[r] & \cdots \\
    }\]
    Upon taking colimit over $r$ we get an isomorphism:
    \[f_*:\pi_n\uHOM(D,A)\to\pi_n\uHOM(D,B)\]
    This finishes the proof.
\end{proof}

\begin{prop}[cf. \cite{uuye}*{Cor. 2.6}]\label{pullhom}
    For any algebra $D$, the functor $\uHOM(D,-):\Alg\to\sSet$ preserves pullbacks.
\end{prop}
\begin{proof}
    Since $\Exinf$ preserves pullbacks, it suffices to show that $\uhom(D,-)$ preserves pullbacks. But the latter follows from the facts that $\uhom(D,A)_n=\hom_{\Alg}(D,A^{\Delta^n})$ and that $(-)^{\Delta^n}:\Alg\to\Alg$ preserves pullbacks since it is naturally isomorphic to tensoring with the polynomial algebra $\ell[s_1,\dots, s_n]$.
\end{proof}

In \cite{uuye}*{Theorem 2.19} it is proved that $(\CAlg, \FSch, W_h)$ is a category of fibrant objects. One might expect a similar result to hold in the algebraic context---more precisely, that there exists a class of fibrations $\cF$ making $(\Alg, \cF, W_H)$ into a category of fibrant objects.
Natural candidates for path-objects $(FW_2)$ are the diagrams
\begin{equation}\label{eq:pathobject}\xymatrix@C=4em{B\ar[r]^-{\const}_-{\sim} & B_n[t]\ar[r]^-{(\ev_0,\ev_1)} & B\times B}\end{equation}
where $B_n[t]$, $\ev_0$ and $\ev_1$ are defined in section \ref{sec:hah}.
Thus, we want the family $\cF$ to contain the morphisms $\ev=(\ev_0, \ev_1):B_n[t]\to B\times B$.
Mimicking the definition of a Schochet fibration, we could take $\cF$ to be the class of algebra homomorphisms that induce a Kan fibration upon applying $\uHOM(D,-)$ for every algebra $D$. It is easily verified that this family $\cF$ is closed by pullbacks but it does not cointain the morphisms $\ev$, as we proceed to explain. Define:
\begin{align*}
P_0&=\{p\in \ell[t]: p(0)=0\}\subset \ell_0[t] \\
P_1&=\{(p,q)\in \ell[t]\times \ell[t] : p(0)=0, p(1)=q(0)\}\subset \ell_1[t]\\
\Omega_0&=\ker(\ev_1:P_0\to \ell)\\
\Omega_1&=\ker(\ev_1:P_1\to \ell)
\end{align*}
These fit into a commutative diagram of algebras
\begin{equation}\label{eq:sq1}
\begin{gathered}
\xymatrix@C=1.2em@R=1.2em{ & \Omega_0\ar@{->}[rr]\ar@{->}'[d][dd]
&& P_0\ar@{->>}[dd]^(.3){\ev_1}\\
\Omega_1\ar@{<-}[ur]^-{\beta}\ar@{->}[rr]\ar@{->}[dd]
&& P_1\ar@{<-}[ur]^-{\sim}\ar@{->>}[dd]^(.3){\ev_1}\\
& 0\ar@{->}'[r][rr] & & \ell\\
0\ar@{->}[rr]\ar@{=}[ur]
&& \ell\ar@{=}[ur]}
\end{gathered}
\end{equation}
where the front and back faces are pullbacks, $\beta$ is the morphism of Lemma \ref{betano}, and $P_0\overset{\sim}{\to} P_1$ is a homotopy equivalence since both $P_0$ and $P_1$ are contractible algebras.
Suppose that the morphisms $\ev$ belong to $\cF$. Since $\cF$ is closed by pullbacks, the right vertical morphisms in \eqref{eq:sq1} lie also in $\cF$.
By Proposition \ref{indhom} and Proposition \ref{pullhom} we get the diagram of Kan complexes
    \begin{equation}\label{eq:cuadradoHOM}\begin{gathered}
\xymatrix@C=-2em@R=1.2em{ & \uHOM(D,\Omega_0)\ar@{->}[rr]\ar@{->}'[d][dd]
&& \uHOM(D, P_0)\ar@{->>}[dd]\\
\uHOM(D,\Omega_1)\ar@{<-}[ur]^(.4){\beta_*}\ar@{->}[rr]\ar@{->}[dd]
&& \uHOM(D,P_1)\ar@{<-}[ur]^-{\sim}\ar@{->>}[dd]\\
& \star \ar@{->}'[r][rr] & & \uHOM(D,\ell)\\
\star \ar@{->}[rr]\ar@{=}[ur]
&& \uHOM(D,\ell) \ar@{=}[ur]}\end{gathered}\end{equation}
where the front and back faces are homotopy pullbacks.
It follows that $\beta_*$ is a weak equivalence of simplicial sets. Then $\beta: \Omega_0\to \Omega_1$ is a polynomial homotopy equivalence by Proposition \ref{indhom}, contradicting Lemma \ref{betano}. This shows that $\cF$ does not contain the morphisms $\ev$. In fact, almost the same argument shows that there is no family of fibrations $\cF$ that contains the morphisms $\ev$ and makes $(\Alg, \cF, W_H)$ into a category of fibrant objects.
\begin{prop}\label{prop:algfnotcof} Let $\ell$ be a domain with $\KH_0(\ell)\ne 0$. Let $W_H$ be the set of polynomial homotopy equivalences in $\Alg$. Let $\cF$ be a set of morphisms in $\Alg$ that is closed by pullbacks and that contains the evaluation map $\ev:\ell[t]\to \ell\times \ell$ defined by $\ev(p) = (p(0), p(1))$. Then 
$(\Alg, \cF, W_H)$ is NOT a category of fibrant objects.
\end{prop} 
\begin{proof}
Consider the commutative diagram of algebras \eqref{eq:sq1}. Since $\ev\in\cF$ and $\cF$ is closed by pullbacks, then the right vertical morphisms in \eqref{eq:sq1} lie in $\cF$ as well. If $(\Alg, \cF, W_H)$ was a category of fibrant objects, then $\beta$ would be a homotopy equivalence by the co-glueing lemma \cite{goja}*{Lemma 9.10}. But $\beta$ is not a homotopy equivalence by Lemma \ref{betano}. It follows that $(\Alg, \cF, W_H)$ is not a category of fibrant objects.
\end{proof}

Let $\Fib$ be the set of those surjective morphisms in $\Alg$ that have a linear section. By Proposition \ref{prop:algfnotcof},
$(\Alg, \Fib, W_H)$ is not a category of fibrant objects---all axioms are satisfied except for $(FW_1)$.
However, we do get a category of fibrant objects if we replace polynomial homotopy equivalences by $kk$-equivalences, as shown below.

\begin{prop}\label{prop:algcof} 
    Let $W_{kk}$ be the set of $\kk$-equivalences and let $\Fib$ be the set of those surjective morphisms in $\Alg$ that have a linear section. Then $(\Alg, \Fib, W_{kk})$ is a category of fibrants objects. 
\end{prop}
\begin{proof}
Most axioms are straightforward to verify. Let us show that fibrations and trivial fibrations are preserved by pullbacks.
Consider a \emph{Milnor square} of algebras, i.e.\ a pullback square as follows where $f$ is surjective and has a linear section:
\[
    \xymatrix{A\ar[r]^{\overline{g}}\ar[d]_-{\overline{f}} & B\ar[d]^-{f} \\
    C\ar[r]_{g} & D}
\]
Since this square is a pullback of $\ell$-modules, it follows from its universal property that $\overline{f}$ also belongs to $\Fib$. 
   \[\xymatrix{
   \save []+<0.5em,-0.5em>*+{C}="C"\restore & & \\
		& A \ar@{->>}[d]_-{\overline{f}}\ar[r]^-{\overline{g}} &
		B \ar@{->>}[d]^-{f} \\
		& C \ar[r]_-{g} &
		D \ar@/_1pc/[u]_{s}
            \ar@/^/"C";"2,3"^-{s\circ g}
            \ar@/_/"C";"3,2"_-{\id}
            \ar@{-->}"C";"2,2"
		}\]
Suppose now that $f\in W_\kk$. For each algebra $E$, we have a long exact Mayer-Vietoris sequence as follows \cite{er}*{Lemma B.5}, where $f^*:\kk_*(D,E)\to \kk_*(B,E)$ is an isomorphism:
\[\xymatrix@C=3em@R=2em{
\save []+<0em, 0.2em>*+{\vdots}="A"\restore &&& \\
\kk_*(D,E) \save []+<0em, -1.5em>*+{(\ast)}="1"\restore \ar[r]^-{\begin{pmatrix}
    \scriptstyle f^* \\ \scriptstyle g^*
\end{pmatrix}}& \kk_*(B,E)\oplus \kk_*(C,E) \save []+<0em, -1.5em>*+{(\ast\ast)}="2"\restore \ar[r]^-{\scriptstyle(\overline{g}^*\; -\overline{f}^*
)} & \kk_*(A,E)\ar[r]^-{\partial} & \kk_{*-1}(D,E)\\
&&& \save []+<0em, 0.2em>*+{\vdots}="Z"\restore
\ar@{->}"A";"2,1"
\ar@{->}"2,4";"Z"+<0em,0.6em>
}\]
Let $x\in \ker\overline{f}^*$. By the exactness at $(\ast\ast)$, there exists $z\in \kk_*(D,E)$ such that
$(f^*(z),g^*(z))=(0,x)$. Since $f^*$ is injective, then $z=0$ and $x=0$. This shows that $\overline{f}^*$ is injective.

It follows from the exactness at $(\ast)$ and the injectivity of $f^*$ that $\partial$ is the zero morphism. Hence, $(\overline{g}^*\; -\overline{f}^*
)$ is surjective.

Let $w\in  \kk_*(A,E)$. There exists a pair $(u,v)$ such that:
\[
w=\overline{g}^*(u)-\overline{f}^*(v) = \overline{g}^*(f^*(t))-\overline{f}^*(v)= \overline{f}^*(g^*(t))-\overline{f}^*(v)  = \overline{f}^*( g^*(t) -v )
\]
This shows that $\overline{f}^*$ is surjective. Then $\overline{f}^*$ is an isomorphism and $\overline{f}\in W_\kk$.
\end{proof}

\begin{rem} Let $\cC$ be a pointed category of fibrant objects.
By \cite{brown}*{Theorem 3} there is a functor $\Omega: \ho(\cC)\to\ho(\cC)$ such that, for any object $B$ and any path object $B^I$, $\Omega B$ is canonically identified with the fibre of $B^I\to B\times B$. Furthermore, $\Omega B$ has a natural group structure. The \emph{stable homotopy category} of $\cC$ is, by definition, the category $\ho(\cC)[\Omega^{-1}]$ obtained from $\ho(\cC)$ upon inverting this endofunctor $\Omega$; see \cite{uuye}*{Def. 1.23}.
We say that $\cC$ is \emph{stable} if $\Omega: \ho(\cC)\to\ho(\cC)$ is an isomorphism. 

Endow $\Alg$ with the category of fibrant objects structure of Proposition \ref{prop:algcof}.
A path object for an algebra $B$ is the diagram \eqref{eq:pathobject} for any $n$. The loop endofunctor mentioned above is induced by $\Omega: \Alg \to \Alg$, $\Omega B= (t^2-t)B[t]$. 
We will see in Theorem \ref{kkuniprop} that $\ho(\Alg)\cong \kk$.
This implies that this category of fibrant objects is stable since $\Omega: \kk \to \kk$ is an isomorphism \cite{cortho}*{Lemma 6.3.8}.
\end{rem}

\section{Algebraic \texorpdfstring{$\kk$}{kk}-theory as a localization of categories}

In this section we prove that the functor $j:\Alg\to\kk$ is the localization of $\Alg$ at the set of $\kk$-equivalences in the sense of ordinary categories.
Consider the following sets of morphims in $\Alg$:
\begin{itemize}
    \item $W_H$, the set of homotopy equivalences;
    \item $W_S:=\{\iota_A:A\to \Minf A: A\in\Alg\}$, the set of upper-left corner inclusions;
    \item $W_E$, the set of classifying maps of extensions
    \[\xymatrix{A\ar[r] & B\ar[r] & C}\]
    where $B$ is either a contractible ring or an infinite sum ring.
\end{itemize}
Write $\Iso(\cC)$ for the set of isomorphisms in a category $\cC$. The main technical result of this section is the following.
\begin{thm}[cf. \cite{bel}*{Proposition 2.1}, \cite{ln}*{Theorem 2.3}]\label{thm:kkloc} 
    The functor $j:\Alg\to\kk$ is initial in the category of those functors $F:\Alg\to \cC$ such that:
    \begin{itemize}
        \item $\cC$ is a category with finite products and $F$ preserves finite products;
        \item $F(W_H\cup W_S\cup W_E)\subseteq\Iso(\cC)$.
    \end{itemize}
\end{thm}
The idea of the proof is simple. Suppose that we have a functor $F:\Alg\to \cC$ satisfying the conditions stated in Theorem \ref{thm:kkloc}. We want to show that $F$ factors uniquely through some $\tF:\kk\to\cC$. It is easy to define $\tF$ on objects: since $j$ is the identity on objects we should have $\tF(A)=F(A)$ for all $A$. To define $\tF$ on morphisms, recall that any morphism $\phi:A\to B$ in $\kk$ is represented by some algebra homomorphism $f:J^vA\to M_kM_\infty B^{S^v}_r$ \cite{cortho}*{Sec. 6.1}. Moreover, as we explain below, in $\cC$ we have identifications $F(\Sigma^v J^vA)\cong F(A)$ and $F(\Sigma^vM_kM_\infty B^{S^v}_r)\cong F(B)$. We would then define $\tF(\phi)$ as the composite:
\[\xymatrix{F(A)\cong F(\Sigma^v J^vA)\ar[r]^-{F(\Sigma^vf)} & F(\Sigma^vM_kM_\infty B^{S^v}_r)\cong F(B)}\]
The rest of this section is devoted to showing that this definition indeed makes sense and gives a well defined functor $\tF:\kk\to\cC$.

\subsection{Excision and \texorpdfstring{$F$}{F}-equivalences}

\begin{defi}\label{defi:Fequiv}
    Let $F:\Alg\to\cC$ be a functor. We say that an algebra homomorphism is an \emph{$F$-equivalence} if it becomes an isomorphism upon applying $F(C\otimes -)$, for every algebra $C$.
\end{defi}

\begin{exa}\label{exa:rhoFequiv}
    Let $F:\Alg\to\cC$ be a functor such that $F(W_E)\subseteq \Iso(\cC)$ and let $A$ be an algebra. Let $\rho_A:JA\to A^{S^1}$ be the classifying map of the loop extension of $A$ \cite{cortho}*{Sec. 4.5}. Then $\rho_A$ is an $F$-equivalence. Indeed, for any algebra $C$, we have a morphism of extensions
    \[\xymatrix{C\otimes JA\ar[d]_-{C\otimes \rho_A}\ar[r] & C\otimes TA\ar[d]\ar[r] & C\otimes A\ar@{=}[d] \\
    C\otimes A^{S^1}\ar[r] & C\otimes PA\ar[r] & C\otimes A}\]
    that induces a commutative square
    \[\xymatrix{J(C\otimes A)\ar@{=}[d]\ar[r]^-{\clas} & C\otimes JA\ar[d]^-{C\otimes \rho_A}\\
    J(C\otimes A)\ar[r]^-{\clas} & C\otimes A^{S^1}}\]
    where the horizontal morphisms belong to $W_E$ and thus become isomorphisms upon applying $F$.
\end{exa}

\begin{lem}\label{lem:JFequiv}
    Let $F:\Alg\to\cC$ be a functor such that $F(W_E)\subseteq \Iso(\cC)$. If $f:A\to B$ is an $F$-equivalence, then $J(f):JA\to JB$ is an $F$-equivalence too.
\end{lem}
\begin{proof}
    Consider the following commutative diagram:
    \[\xymatrix@C=4em{F(C\otimes JA)\ar[d]_-{F(C\otimes J(f))}\ar[r]^-{F(C\otimes \rho_{JA})} & F(C\otimes A^{S^1})\cong F((\ell^{S^1}\otimes C)\otimes A)\ar@<3em>[d]^-{F((\ell^{S^1}\otimes C)\otimes f)} \\
    F(C\otimes JB)\ar[r]^-{F(C\otimes \rho_{JB})} & F(C\otimes B^{S^1})\cong F((\ell^{S^1}\otimes C)\otimes B)}\]
    The right vertical morphism is an isomorphism since $f$ is an $F$-equivalence and both horizontal morphisms are isomorphisms by Example \ref{exa:rhoFequiv}. Then the left vertical morphism is an isomorphism too.
\end{proof}

\begin{lem}\label{lem:loopingFequiv}
    Let $F:\Alg\to\cC$ be a functor such that $F(W_E)\subseteq \Iso(\cC)$. If $f$ is an algebra homomorphism that fits into a morphism of extensions
    \begin{equation}\label{eq:loppingFequiv}\begin{gathered}\xymatrix{A\ar[r]\ar[d]_-{f} & A'\ar[d]\ar[r] & A''\ar[d]^-{f''} \\
    B\ar[r] & B'\ar[r] & B''}\end{gathered}\end{equation}
    where $f''$ is an $F$-equivalence and $A'$ and $B'$ are either contractible rings or infinite sum rings, then $f$ is an $F$-equivalence. In particular, the classifying maps of extensions where the middle term is either a contractible ring or an infinite sum ring are $F$-equivalences.
\end{lem}
\begin{proof}
    For any algebra $C$, the morphism \eqref{eq:loppingFequiv} induces a commutative square:
    \[\xymatrix{F(J(C\otimes A''))\ar[d]_-{F(J(C\otimes f''))}\ar[r]^-{F(\clas)} & F(C\otimes A)\ar[d]^-{F(C\otimes f)} \\
    F(J(C\otimes B''))\ar[r]^-{F(\clas)} & F(C\otimes B)}\]
    The left vertical morphism is an isomorphism by Lemma \ref{lem:JFequiv}, since $f''$ -- and thus $C\otimes f''$ -- is an $F$-equivalence. The horizontal morphisms are isomorphisms since they result from applying $F$ to morphisms in $W_E$. Then the right vertical morphism is an isomorphism too.
\end{proof}

\begin{lem}\label{lem:Fequivs}
    Let $F:\Alg\to\cC$ be a functor such that $F(W_E)\subseteq \Iso(\cC)$. Then the following algebra homomorphisms are $F$-equivalences for any algebra $A$:
    \begin{enumerate}
        \item\label{item:FequivsX} $x_A:J\Sigma A\to \Minf A$, the classifying map of the cone extension of $A$; see \cite{cortho}*{Sec. 4.7};
        \item $c_A:J\Sigma A\to \Sigma JA$, the classifying map of the extension
        \[\xymatrix{\Sigma JA\ar[r] & \Sigma TA\ar[r] & \Sigma A};\]
        \item\label{item:FequivsLVM} $\gamma^n_A: A^{S^n}_r\to A^{S^n}_{r+1}$, the morphism induced by the last vertex map;
        \item $\mu^{m,n}_A:(A^{S^m}_r)^{S^n}_s\to A^{S^{m+n}}_{r+s}$, the morphism defined in \cite{htpysimp}*{Sec. 3.1}.
    \end{enumerate}
\end{lem}
\begin{proof}
    To see that $c_A$ and $x_A$ are $F$-equivalences, note that they fit into the following morphisms of extensions and use Lemma \ref{lem:loopingFequiv}.
    \[\xymatrix{ J\Sigma A\ar[d]_-{ c_A}\ar[r] &  T\Sigma A\ar[d]\ar[r] &  \Sigma A\ar@{=}[d] & & J\Sigma A\ar[d]_-{ x_A}\ar[r] &  T\Sigma A\ar[d]\ar[r] &  \Sigma A\ar@{=}[d] \\
     \Sigma JA\ar[r] &  \Sigma TA\ar[r] &  \Sigma A & & \Minf A\ar[r] &  \Gamma A\ar[r] &  \Sigma A}\]

    To prove that $\gamma^n_A$ is a weak equivalence we will proceed by induction on $n$. For $n=0$ the result is obvious since $A^{S^0}_r\cong A$ for any $r$ and $\gamma^0_A$ is identified with the identity of $A$. For the inductive step, recall from \cite{loopstho}*{Sec. 2.28} that we have a morphism of extensions 
    \[\xymatrix{ A^{S^{n+1}}_r\ar[d]_-{\gamma^{n+1}_A}\ar[r] &  P(n,A)_r\ar[d]\ar[r] &  A^{S^n}_r\ar[d]^-{\gamma^n_A}\\
     A^{S^{n+1}}_{r+1}\ar[r] &  P(n,A)_{r+1}\ar[r] &  A^{S^n}_{r+1}}\]
    where the middle terms are contractible. If $\gamma^n_A$ is an $F$-equivalence, then $\gamma^{n+1}_A$ is an $F$-equivalence as well by Lemma \ref{lem:loopingFequiv}.
    
    Let us prove that $\mu^{m,n}_A$ is an $F$-equivalence. First note that it suffices to consider the case $r=s=0$ since we have a commutative square
    \[\xymatrix@C=4em{(A^{S^m}_r)^{S^n}_s\ar[r]^-{\mu^{m,n}_A} & A^{S^{m+n}}_{r+s} \\
    (A^{S^m}_0)^{S^n}_0\ar[u]\ar[r]^-{\mu^{m,n}_A} & A^{S^{m+n}}_0\ar[u]}\]
    where the vertical morphisms are $F$-equivalences by \eqref{item:FequivsLVM}. We will proceed by induction on $n$. For $n=0$ there is nothing to prove, since $(A^{S^m})^{S^0}\cong A^{S^m}$ and $\mu^{m,0}_A$ identifies with the identity of $A^{S^m}$. For the inductive step, recall from \cite{loopstho}*{Example 2.29} that we have a morphism of extensions
    \[\xymatrix{(A^{S^m})^{S^{n+1}}\ar[d]_-{\mu^{m,n+1}_A}\ar[r] & P(n, A^{S^m})\ar[d]\ar[r] & (A^{S^m})^{S^n}\ar[d]^-{\mu^{m,n}_A} \\
    A^{S^{m+n+1}}\ar[r] & P(m+n,A)\ar[r] & A^{S^{m+n}}}\]
    where the middle terms are contractible. If $\mu^{m,n}_A$ is an $F$-equivalence, then $\mu^{m, n+1}_A$ is an $F$-equivalence as well by Lemma \ref{lem:loopingFequiv}.
\end{proof}

\subsection{Group objects in the localization}
\begin{lem}\label{lem:comm}
    Let $\cC$ be a category with finite products and let $F:\Alg\to \cC$ be a functor that preserves finite products and such that $F(W_H\cup W_S)\subseteq \Iso(\cC)$. Then:
    \begin{enumerate}
        \item\label{item:comm1} for every algebra $A$, $F(A)$ is a commutative monoid object in $\cC$;
        \item\label{item:comm2} for every algebra homomorphism $f:A\to B$, $F(f):F(A)\to F(B)$ is a morphism of monoid objects in $\cC$.
    \end{enumerate}
\end{lem}
\begin{proof}
    Let $A$ be an algebra and let $\bar{m}_A:A\times A \to M_2(A)$ be defined by
    \[(a_1, a_2)\mapsto \begin{pmatrix}
        a_1 & 0 \\ 0 & a_2
    \end{pmatrix}.\]
    Let $m_A:F(A)\times F(A)\to F(A)$ be the following composite in $\cC$:
    \[\xymatrix@C=3em{F(A)\times F(A)\cong F(A\times A)\ar[r]^-{F(\bar{m}_A)} & F(M_2A)\ar[r]^-{F(\iota_A)^{-1}}_-{\cong} & F(A)}\]
    Note that $F(0)$ is a terminal object in $\cC$ since $F$ preserves products. Let $u_A:F(0)\to F(A)$ be induced by the zero morphism $0\to A$. We claim that $m_A$ and $u_A$ are, respectively, a multiplication and a unit that make $F(A)$ into a monoid object.
    Let us prove that $m_A$ is associative, i.e. that the following diagram in $\cC$ commutes:
    \[\xymatrix@C=4em{F(A)\times F(A)\times F(A)\ar[d]_-{m_A\times \id}\ar[r]^-{\id\times m_A} & F(A)\times F(A) \ar[d]^-{m_A} \\
    F(A)\times F(A)\ar[r]^-{m_A} & F(A)}\]
    Unravelling the definition of $m_A$, one shows that the commutativity of the square above is equivalent to that of the outer square in the following diagram:
    \begin{equation}\begin{gathered}\label{eq:comm}\xymatrix@C=3em{F(A\times A\times A)\ar[d]_-{F(\bar{m}_A\times \id)}\ar[r]^-{F(\id\times \bar{m}_A)} & F(A\times M_2A)\ar[r]^-{F(\id\times \iota_A)^{-1}}\ar[d]|-{F(\iota_A\times \id)} & F(A\times A) \ar[d]^-{F(\bar{m}_A)} \\
    F(M_2A\times A)\ar[r]^-{F(\id\times\iota_A)}\ar[d]_-{F(\iota_A\times\id)^{-1}} & F(M_2A\times M_2A)\ar[dr]|-{F(\bar{m}_{M_2A})}\ar@{}[r]^-{(\star)} & F(M_2A)\ar[d]^-{F(\iota_{M_2A})} \\
    F(A\times A)\ar[r]_-{F(\bar{m}_A)} & F(M_2A) \ar[r]_-{F(\iota_{M_2A})} & F(M_4A)}\end{gathered}\end{equation}
    Suppose for a moment that both trapezoids in \eqref{eq:comm} commute. Then, to prove that the outer square in \eqref{eq:comm} commutes it suffices to show that
    $\bar{m}_{M_2A}\circ (\iota_A\times \id)\circ (\id\times\bar{m}_A)$ and $\bar{m}_{M_2A}\circ (\id\times\iota_A)\circ (\bar{m}_A\times \id)$ become equal upon applying $F$. The latter is easily verified, as we proceed to explain. We have
    \begin{align*}
        [\bar{m}_{M_2A}\circ (\iota_A\times \id)\circ (\id\times\bar{m}_A)](a_1, a_2, a_3)&=\begin{pmatrix}
            a_1 & 0 & 0 & 0 \\ 0 & 0 & 0 & 0 \\ 0 & 0 & a_2 & 0 \\ 0 & 0 & 0 & a_3
        \end{pmatrix}, \\
        [\bar{m}_{M_2A}\circ (\id\times\iota_A)\circ (\bar{m}_A\times \id)](a_1, a_2, a_3)&=\begin{pmatrix}
            a_1 & 0 & 0 & 0 \\ 0 & a_2 & 0 & 0 \\ 0 & 0 & a_3 & 0 \\ 0 & 0 & 0 & 0
        \end{pmatrix},
    \end{align*}
    and both matrices are conjugate by a permutation matrix. Thus, they induce the same morphism in $\cC$ by \cite{friendly}*{Proposition 2.2.6}. To show that both trapezoids in \eqref{eq:comm} commute one can proceed in a similar fashion. For example, the trapezoid marked with $(\star)$ commutes since we have
    \begin{align*}[\iota_{M_2A}\circ \bar{m}_A](a_1,a_2)=\begin{pmatrix}
        a_1 & 0 & 0 & 0 \\ 0 & a_2 & 0 & 0 \\ 0 & 0 & 0 & 0 \\ 0 & 0 & 0 & 0
    \end{pmatrix},\\
    [\bar{m}_{M_2A}\circ (\iota_A\times \id)\circ (\id\times\iota_A)](a_1,a_2)=\begin{pmatrix}
        a_1 & 0 & 0 & 0 \\ 0 & 0 & 0 & 0 \\ 0 & 0 & a_2 & 0 \\ 0 & 0 & 0 & 0
    \end{pmatrix},\end{align*}
    and both matrices are conjugate by a permutation matrix. This proves the associativity of $m_A$. The commutativity of $m_A$ and the fact that $u_A$ is a unit are easily verified. This finishes the proof of \eqref{item:comm1}.
    For \eqref{item:comm2}, note that $m_A$ and $u_A$ are natural in $A$ with respect to algebra homomorphisms since both $\bar{m}_A$ and $\iota_A$ are.
\end{proof}


\begin{lem}\label{lem:comg}
    Let $\cC$ be a category with finite products and let $F:\Alg\to \cC$ be a functor that preserves finite products and such that $F(W_H\cup W_S\cup W_E)\subseteq\Iso(\cC)$. Then:
    \begin{enumerate}
        \item\label{item:comg1} $F(A)$ is a commutative group object in $\cC$ for every $A\in\Alg$ and moreover this group structure coincides with the monoid structure from Lemma \ref{lem:comm};
        \item\label{item:comg2} $F(f):F(A)\to F(B)$ is a morphism of group objects in $\cC$ for every morphism $f:A\to B$ in $\Alg$.
        \item\label{item:comg3} The function $F^\ind(C\otimes -):[A,B^{S^1}_\bullet]\to \hom_\cC(F(C\otimes A), F(C\otimes B^{S^1}_\bullet))$ is a group homomorphism for any $A,B,C\in\Alg$.
    \end{enumerate}
\end{lem}
\begin{proof}
    Note that \eqref{item:comg2} will follow from Lemma \ref{lem:comm} \eqref{item:comm2} once we prove \eqref{item:comg1}.

    To prove \eqref{item:comg1}, recall from Lemma \ref{lem:Fequivs} \eqref{item:FequivsLVM} that we have an isomorphism $F(\gamma_A): F(A^{S^1}_r)\to F(A^{S^1}_{r+1})$ for every $r\geq 0$. This implies that $F(A^{S^1})$ and $F(A^{S^1}_\bullet)$ are isomorphic objects of $\cC^\ind$.
    It is well known that $A^{S^1}_\bullet$ is a group object in $[\Alg]^\ind$; see \cite{cortho}*{Thm. 3.3.2} and \cite{htpysimp}*{Thm. 3.10}. Upon applying $F^\ind$, we get that $F(A^{S^1}_\bullet)\cong F(A^{S^1})$ is a group object in $\cC^\ind$. Hence $F(A^{S^1})$ is a group object in $\cC$. We claim that the group operation in $F(A^{S^1})$ coincides with the commutative monoid operation from Lemma \ref{lem:comm}. Indeed, the group operation in $F(A^{S^1})$ can be described in terms of concatenation \cite{loopstho}*{Ex. 2.20} as the following composite:
    \[\xymatrix@C=4em{F(A^{S^1})\times F(A^{S^1})\cong F(A^{S^1}\times A^{S^1})\ar[r]^-{F(\concat)} & F(A^{S^1}_1) \ar[r]^-{F(\gamma_A)^{-1}}_-{\cong} & F(A^{S^1})}\]
    It is then clear that this composite is a morphism of monoid objects since both $F(\concat)$ and $F(\gamma_A)$ are by Lemma \ref{lem:comm} \eqref{item:comm2}. It follows that both operations coincide by the Hilton-Eckmann argument. This shows that the monoid structure of $F(A^{S^1})$ from Lemma \ref{lem:comm} is indeed a commutative group structure.

    Now consider the following chain of isomorphisms in $\cC$:
    \[\xymatrix{F(A)\ar[r]^-{F(\iota_A)}_-{\cong} & F(\Minf A) & F(J\Sigma A)\ar[l]^-{\cong}_-{F(x_A)}\ar[r]_-{\cong}^-{F(\rho_{\Sigma A})} & F((\Sigma A)^{S^1})}\]
    By Lemma \ref{lem:comm} \eqref{item:comm2}, Example \ref{exa:rhoFequiv} and Lemma \ref{lem:Fequivs} \eqref{item:FequivsX}, each of these is a monoid isomorphism for the monoid structure from Lemma \ref{lem:comm}. We have shown that the monoid structure of $F((\Sigma A)^{S^1})$ is indeed a commutative group structure. It follows that the monoid structure of $F(A)$ is a commutative group structure as well.

    The assertion \eqref{item:comg3} follows from the above and \cite{loopstho}*{Proposition 2.46 (ii)}.
\end{proof}

\subsection{An explicit construction of the \texorpdfstring{$\kk$}{kk}-theory category}
Following \cite{tesisema}, we proceed to recall one way to construct the algebraic $\kk$-theory category defined by Corti\~nas-Thom in \cite{cortho}. The construction goes in three steps:
\begin{enumerate}
    \item Construct a triangulated category $\fk$ whose objects are the pairs $(A,m)$ with $A\in\Alg$ and $m\in\Z$, and whose morphism sets are defined by:
    \[\hom_{\fk}((A,m),(B,n))=\colim_{v}[J^{m+v}A, B^{S^{n+v}}_\bullet]\]
    See \cite{loopstho}*{Definition 3.4} for more details. There is a functor $j:\Alg\to\fk$, $j(A)=(A,0)$, that is the universal excisive and homotopy invariant homology theory in the sense of \cite{cortho}*{Section 6.6}; see \cite{loopstho}*{Theorem 10.16} for a proof of this statement. This category $\fk$ is equivalent to the one defined by Garkusha in \cite{garkuni}*{Thm. 2.6} since they both share the same universal property.
    \item Construct a triangulated category $\kf$ whose objects are those of $\fk$ and whose morphism sets are defined by:
    \[\hom_{\kf}((A,m),(B,n))=\colim_{p}\hom_{\fk}((A,m),(M_pB, n))\]
    Here, the transition maps are induced by the upper-left corner inclusions $M_pB\to M_{p+1}B$; see \cite{tesisema}*{Definition 5.1.3} for more details on the definition of $\kf$. There is a functor $j_f:\Alg\to\kf$, $j_f(A)=(A,0)$, that is the universal excisive, homotopy invariant and $M_n$-stable homology theory; see \cite{tesisema}*{Theorem 5.1.12} for a proof of this statement. This category $\kf$ is equivalent to the one defined by Garkusha in \cite{garkuni}*{Thm. 6.5} since they both share the same universal property.
    \item Construct a triangulated category $\ks$ whose objects are those of $\kf$ and whose morphism sets are defined by:
    \[\hom_{\ks}((A,m),(B,n))=\hom_{\kf}((A,m),(M_\infty B, n))\]
    See \cite{tesisema}*{Definition 5.2.7} for more details. There is a functor $j_s:\Alg\to\ks$, $j_s(A)=(A,0)$, that is the universal excisive, homotopy invariant and $M_\infty$-stable homology theory; see \cite{tesisema}*{Theorem 5.2.16} for a proof of this statement. This category $\ks$ is equivalent to the category $\kk$ defined by Corti\~nas-Thom \cite{cortho} and to the category defined by Garkusha in \cite{garkuni}*{Theorem 9.3} since they all share the same universal property.
\end{enumerate}

\begin{rem}
    The translation functors of the triangulated categories $\fk$, $\kf$ and $\ks$ are given on objects by $(A,n)\mapsto (A, n+1)$. We have a commutative diagram \begin{equation}\label{eq:allk}\begin{gathered}\xymatrix{\fk\ar[r]^-{t_f} & \kf\ar[r]^-{t_s} & \ks \\
    & \Alg\ar@/^1pc/[ul]_-{j}\ar[u]_-{j_f}\ar@/_1pc/[ur]^-{j_s} &}\end{gathered}\end{equation}
    where the horizontal functors are triangulated and are the identity on objects.
\end{rem}

\subsection{Main theorem of the section}

We are now ready to prove Theorem \ref{thm:kkloc} (cf. Theorem \ref{thm:kklocposta}). For the rest of this section, we fix a category $\cC$ with finite products and a functor $F:\Alg\to\cC$ that preserves finite products and such that $F(W_H\cup W_S\cup W_E)\subseteq \Iso(\cC)$. We start by making precise how to identify $F(\Sigma^vJ^vA)\cong F(A)$ (Def. \ref{defi:alpha}) and $F(\Sigma^v A^{S^v})\cong F(A)$ (Def. \ref{defi:beta}).

\begin{defi}\label{defi:alpha}
    For $A\in\Alg$ and $n,k\geq 0$ we will define isomorphisms in $\cC$:
\[\xymatrix{\alpha^{n,k}_A:F(\Sigma^{k+n}J^k A)\ar[r]^-{\cong} & F(\Sigma^nA)}\]
We proceed inductively on $k$. Let $\alpha^{n,0}_A$ be the identity of $F(\Sigma^nA)$. Let $\alpha^{n,1}_A$ be the following composite in $\cC$:
\[
\xymatrix@C=2em{
F(\Sigma^{1+n}JA)=F(\Sigma^n(\Sigma JA))\ar[r]^-{(c_A)^{-1}_*}_-{\cong} & F(\Sigma^n(J\Sigma A))\ar[r]^-{(x_A)_*}_-{\cong} & F(\Sigma^nM_\infty A)\ar[r]^-{(\iota_A)^{-1}_*}_-{\cong} & F(\Sigma^nA)
}
\]
Here, the left and middle morphisms are isomorphisms by Lemma \ref{lem:Fequivs} and the right morphism is an isomorphism by $M_\infty$-stability.
Now suppose that we have defined $\alpha^{n,h}_A$ for all $A\in\Alg$, all $n\geq 0$ and all $1\leq h\leq k$. Let $\alpha^{n, k+1}_A$ be the following composite in $\cC$:
\[\xymatrix{
F(\Sigma^{1+k+n}J^{k+1}A)=F(\Sigma^{1+k+n}J(J^kA))\ar[r]^-{\alpha^{k+n,1}_{J^kA}}_-{\cong} & F(\Sigma^{k+n}J^kA)\ar[r]^-{\alpha^{n,k}_A}_-{\cong} & F(\Sigma^nA)
}\]
This defines $\alpha^{n,k}_A$ for all $A\in\Alg$ and all $n,k\geq 0$.
\end{defi}

\begin{lem}\label{lem:alphapq}
    Let $A\in\Alg$ and $n,p,q\geq 0$. Then the following diagram in $\cC$ commutes:
    \[\xymatrix{
        F(\Sigma^{q+p+n}J^{p+q}A)\ar[r]^-{\alpha^{p+n,q}_{J^pA}}\ar@/_1pc/[dr]_-{\alpha^{n,p+q}_A} & F(\Sigma^{p+n}J^pA)\ar[d]^-{\alpha^{n,p}_A} \\
        & F(\Sigma^n A)
    }\]
\end{lem}
\begin{proof}
    We proceed by induction on $q$. For $q=0$ there is nothing to prove. For $q=1$, the result holds by definition of $\alpha^{n,p+1}_A$; see Definition \ref{defi:alpha}. Now suppose that the result holds for $q\geq 1$ and let us show that it also holds for $q+1$. Consider the following diagram in $\cC$:
    \[\xymatrix{
    & F(\Sigma^{1+q+p+n} J^{p+q+1}A)\ar@/_1pc/[ddl]_-{\alpha^{n, p+q+1}_A}\ar[d]^-{\alpha^{q+p+n,1}_{J^{p+q}A}}\ar@/^1pc/[ddr]^-{\alpha^{p+n,q+1}_{J^pA}} & \\
    & F(\Sigma^{q+p+n}J^{p+q}A)\ar[dl]_-{\alpha^{n,p+q}_A}\ar[dr]^-{\alpha^{p+n,q}_{J^pA}} & \\
    F(\Sigma^nA) & & F(\Sigma^{p+n}J^pA)\ar[ll]_-{\alpha^{n,p}_A}
    }\]
    The two upper triangles commute by the case $q=1$ and the lower triangle commutes by the inductive hypothesis. Then the outer triangle commutes as well, proving the result for $q+1$.
\end{proof}

\begin{defi}\label{defi:beta}
    For $A\in\Alg$ and $n,k\geq 0$ we will define isomorphisms in $\cC$:
\[\xymatrix{\beta^{n,k}_A:F(\Sigma^{k+n} A^{S^k})\ar[r]^-{\cong} & F(\Sigma^nA)}\]
We proceed inductively on $k$. Let $\beta^{n,0}_A$ be the identity of $F(\Sigma^nA)$. Let $\beta^{n,1}_A$ be the following composite in $\cC$:
\[\xymatrix{
F(\Sigma^{1+n}A^{S^1})\ar[r]^{(\rho_A)^{-1}_*}_-{\cong} & F(\Sigma^{1+n}JA)\ar[r]^{\alpha^{n,1}_A}_-{\cong} & F(\Sigma^nA)
}\]
Here, the morphism on the left is an isomorphism by Example \ref{exa:rhoFequiv}. Now suppose that we have defined $\beta^{n,h}_A$ for all $A\in\Alg$, all $n\geq 0$ and all $1\leq h\leq k$. Let $\beta^{n, k+1}_A$ be the following composite in $\cC$:
\[\xymatrix{
F(\Sigma^{1+k+n}A^{S^{k+1}})\ar[r]^-{(\mu^{k,1}_A)^{-1}_*}_-{\cong} & F(\Sigma^{1+k+n}(A^{S^k})^{S^1})\ar[r]^-{\beta^{k+n,1}_{A^{S^k}}}_-{\cong} & F(\Sigma^{k+n}A^{S^k})\ar[r]^-{\beta^{n,k}_A}_-{\cong} & F(\Sigma^nA)
}\]
Here, the morphism on the left is an isomorphism by Lemma \ref{lem:Fequivs}. This defines $\beta^{n,k}_A$ for all $A\in\Alg$ and all $n,k\geq 0$.
\end{defi}

\begin{lem}\label{lem:betapq}
    Let $A\in\Alg$ and $n,p,q\geq 0$. Then the following diagram in $\cC$ commutes:
    \[\xymatrix{
        F(\Sigma^{q+p+n}(A^{S^p})^{S^q})\ar[r]^-{(\mu^{p,q}_A)_*}\ar[d]_-{\beta^{p+n,q}_{A^{S^p}}} & F(\Sigma^{q+p+n}A^{S^{p+q}})\ar[d]^-{\beta^{n,p+q}_A} \\
        F(\Sigma^{p+n}A^{S^p})\ar[r]^-{\beta^{n,p}_A} & F(\Sigma^nA)
    }\]
\end{lem}
\begin{proof}
    We proceed by induction on $q$. For $q=0$ there is nothing to prove. For $q=1$, the result holds by definition of $\beta^{n,p+1}_A$; see Definition \ref{defi:beta}. Now suppose that the result holds for $q\geq 1$ and let us show that it also holds for $q+1$. Consider the following diagram in $\cC$:
    \[
    \xymatrix@C=-0.25em{
    *+[r]{F(\Sigma^{1+q+p+n}(A^{S^p})^{S^{q+1}})}\ar[rrrrrrrr]^-{(\mu^{p,q+1}_A)_*}\ar[ddd]^-{\beta^{p+n,q+1}_{A^{S^p}}} & & & & & & & & *+[l]{F(\Sigma^{1+q+p+n}A^{S^{p+q+1}})}\ar[ddd]_-{\beta^{n,p+q+1}_A} \\
    & F(\Sigma^{1+q+p+n}((A^{S^p})^{S^q})^{S^1})\ar[ul]_-{(\mu^{q,1}_{A^{S^p}})_*}\ar[rrrrrr]^-{(\mu^{p,q}_A)_*}\ar[d]_-{\beta^{q+p+n,1}_{(A^{S^p})^{S^q}}} & & & & & & F(\Sigma^{1+q+p+n}(A^{S^{p+q}})^{S^1})\ar[d]^-{\beta^{q+p+n,1}_{A^{S^{p+q}}}}\ar[ur]^-{(\mu^{p+q,1}_A)_*} & \\
    & F(\Sigma^{q+p+n}(A^{S^p})^{S^q})\ar[dl]^-{\beta^{p+n,q}_{A^{S^p}}}\ar[rrrrrr]^-{(\mu^{p,q}_A)_*} & & & & & & F(\Sigma^{q+p+n}A^{S^{p+q}})\ar[dr]_-{\beta^{n,p+q}_A} & \\
    *+[r]{F(\Sigma^{p+n}A^{S^p})}\ar[rrrrrrrr]^-{\beta^{n,p}_A} & & & & & & & & *+[l]{F(\Sigma^n A)}
    }
    \]
    The trapezoid on the top commutes by the associativity of $\mu$, the rectangle in the center commutes by the naturality of $\beta$ with respect to algebra homomorphisms, the left and right trapezoids commute by the case $q=1$ and the trapezoid on the bottom commutes by the inductive hypothesis. Then the outer rectangle commutes as well, proving the result for $q+1$.
\end{proof}

\begin{lem}\label{lem:kappaAB}
    Let $A\in\Alg$ and $p,q,r\geq 0$. Then the following diagram in $\cC$ commutes:
    \[\xymatrix{
    F(\Sigma^{p+q+r}J^p(A^{S^q}))\ar[d]_-{\alpha^{q+r,p}_{A^{S^q}}}\ar[rr]^-{(-1)^{pq}(\kappa^{p,q}_A)_*} & & F(\Sigma^{p+q+r}(J^pA)^{S^q})\ar[d]^-{\beta^{p+r,q}_{J^pA}} \\
    F(\Sigma^{q+r}A^{S^q})\ar[dr]_-{\beta^{r,q}_A} & & F(\Sigma^{p+r}J^p A)\ar[dl]^-{\alpha^{r,p}_A} \\
    & F(\Sigma^rA) &
    }\]
    Note that the sign $(-1)^{pq}$ makes sense by Lemma \ref{lem:comg} \eqref{item:comg1}.
\end{lem}
\begin{proof}
We proceed by induction on $q$. For $q=0$ there is nothing to prove. For $q=1$, consider the following diagram, where all the morphisms are isomorphisms:
    \[\xymatrix@R=4em@C=4em{
    F(\Sigma^{p+1+r}J^p(A^{S^1}))\ar[d]_-{\alpha^{1+r,p}_{A^{S^1}}}\ar[rr]^-{(-1)^{p}(\kappa^{p,1}_A)_*} & & F(\Sigma^{p+1+r}(J^pA)^{S^1})\ar[d]^-{\beta^{p+r,1}_{J^pA}} \\
    F(\Sigma^{1+r} A^{S^1})\ar[dr]_-{\beta^{r,1}_A} & 
    \save []+<0em,1.8em>*+{F(\Sigma^{p+1+r}J^{p+1}A)}="A" \restore
    \save []+<0em,-1.5em>*+{F(\Sigma^{1+r} JA)}="B" \restore
    & F(\Sigma^{p+r}J^p A)\ar[dl]^-{\alpha^{r,p}_A} \\
    & F(\Sigma^rA) &
    \ar"B";"3,2"^-{\alpha^{r,1}_A}
    \ar"A";"B"^-{\alpha^{1+r,p}_{JA}}
    \ar"A";"1,1"_-{(\rho_A)_*}
    \ar"B";"2,1"_-{(\rho_A)_*}
    \ar"A";"2,3"^-{\alpha^{p+r,1}_{J^pA}}
    \ar"A";"1,3"^-{(\rho_{J^pA})_*}
    \ar@{}"1,2";"A"|-{\star}
    \ar@{}"1,1";"B"_-{\star\star}
    \ar@{}"B";"2,3"|-{\star\star\star}
    }\]
    We have to prove that the outer diagram commutes. The triangle marked with $\star$ commutes by \cite{loopstho}*{Lemma 4.4} and Lemma \ref{lem:comg} \eqref{item:comg3}. The square marked with $\star\star$ commutes by naturality of $\alpha^{1+r,p}_?$ and the square marked with $\star\star\star$ commutes by Lemma \ref{lem:alphapq}. The remaining two triangles commute by definition of $\beta$ (Def. \ref{defi:beta}). This proves the result for $n=1$.

    Now suppose that the statement of the lemma holds for $q\geq 1$ and let us prove that it also holds for $q+1$. Consider the following diagram in $\cC$, where all the morphisms are isomorphisms. To ease notation we write $N:=p+q+1+r$ and we omit writing the functor $F$, that should be applied everywhere.
    \[\xymatrix@C=4em{
    \save []+<-3em,4em>*+{\Sigma^NJ^p(A^{S^{q+1}})}="A" \restore
    \Sigma^NJ^p((A^{S^q})^{S^1})\ar[d]^-{\alpha^{q+1+r,p}_{(A^{S^q})^{S^1}}}\ar[r]^-{(-1)^p(\kappa^{p,1}_{A^{S^q}})_*} & \Sigma^N(J^p(A^{S^q}))^{S^1}\ar[d]_-{\beta^{p+q+r,1}_{J^p(A^{S^q})}}\ar[r]^-{(-1)^{pq}(\kappa^{p,q}_A)_*} &
    \save []+<3.2em,4em>*+{\Sigma^N(J^pA)^{S^{q+1}}}="B" \restore
    \Sigma^N((J^pA)^{S^q})^{S^1}\ar[d]_-{\beta^{p+q+r,1}_{(J^pA)^{S^q}}} \\
    \Sigma^{q+1+r}(A^{S^q})^{S^1}\ar@{}[ur]_{\star}\ar[dr]^-{\beta^{q+r,1}_{A^{S^q}}} & \Sigma^{p+q+r}J^p(A^{S^q})\ar[r]^-{(-1)^{pq}(\kappa^{p,q}_A)_*}\ar@{}[ur]^{\star\star}\ar[d]^-{\alpha^{q+r,p}_{A^{S^q}}} & \Sigma^{p+q+r}(J^pA)^{S^q} \\
    \save []+<-3em,0em>*+{\Sigma^{q+1+r}A^{S^{q+1}}}="C" \restore
    & \Sigma^{q+r}A^{S^q}\ar[d]^-{\beta^{r,q}_A}\ar@{}[ur]_-{\star} &
    \save []+<3.2em,0em>*+{\Sigma^{p+r}J^pA}="D" \restore \\
    & \Sigma^r A
    \ar"A"+<-2em,-0.8em>;"C"+<-2em,0.8em>^-{\alpha^{q+1+r,p}_{A^{S^{q+1}}}}
    \ar"A";"B"^-{(-1)^{p(q+1)}(\kappa^{p,q+1}_A)_*}
    \ar"B"+<1.8em,-0.8em>;"D"+<1.8em,0.8em>_-{\beta^{p+r,q+1}_{J^pA}}
    \ar"1,1";"A"_-{(\mu^{q,1}_A)_*}
    \ar"1,3";"B"^-{(\mu^{q,1}_{J^pA})_*}
    \ar"2,1";"C"^-{(\mu^{q,1}_A)_*}
    \ar"2,3";"D"_-{\beta^{p+r,q}_{J^pA}}
    \ar"C";"4,2"^-{\beta^{r,q+1}_A}
    \ar"D";"4,2"_-{\alpha^{r,p}_A}
    \ar@{}"C";"3,2"^-{\sharp}
    \ar@{}"1,2";"1,2"+<0em,4em>|-{\sharp\sharp}
    }\]
    We have to prove that the outer diagram commutes. The pentagons marked with $\star$ commute by the inductive hypothesis. The square marked with $\star\star$ commutes by naturality of $\beta^{p+q+r,1}_?$. The square marked with $\sharp$ commutes by definition of $\beta$. The pentagon marked with $\sharp\sharp$ commutes by \cite{loopstho}*{Lemma 2.38}. It remains to verify that two trapezoids commute. The one on the left commutes by naturality of $\alpha^{q+1+r,p}_?$ and the one on the right commutes by Lemma \ref{lem:betapq}. This finishes the proof.
\end{proof}

\begin{rem}
    Let $\uks$ be the full subcategory of $\ks$ whose objects are the pairs $(A,0)$ with $A\in\Alg$. We claim that the inclusion $\uks\subseteq \ks$ is an equivalence of categories. We have to show that every object of $\ks$ is isomorphic to one of $\uks$. For $n> 0$, we have
    \[(A,n)\cong (J^nA,0)\]
    in $\fk$ (and thus in $\ks$) by \cite{loopstho}*{Lemma 7.2}. Moreover, we have
    \[(A,-n)\cong (J^n\Sigma^nA, -n)\cong (\Sigma^nA, 0)\]
    where the isomorphism on the right holds by \cite{loopstho}*{Lemma 7.2} and the one on the left is induced by $\alpha^{0,n}_A:(\Sigma^nJ^nA,0)\to(A,0)$ --- take $F=j_s:\Alg\to\ks$. It follows that $\uks$ is equivalent to $\ks$ and, thus, to the category $\kk$ defined by Corti\~nas-Thom in \cite{cortho}.
\end{rem}

\begin{thm}\label{thm:kklocposta}
    Let $\cC$ be a category with finite products and let $F:\Alg\to\cC$ be a functor that commutes with finite products and such that $F(W_H\cup W_S\cup W_E)\subseteq \Iso(\cC)$. Then there exists a unique functor $\tF$ making the following triangle commute:    \begin{equation}\label{eq:UPufk}\begin{gathered}\xymatrix@R=2em{\Alg\ar[r]^-{F}\ar@/_1pc/[dr]_-{j_s} & \cC\\
    & \uks\ar@{-->}[u]_-{\exists!\, \tF}}\end{gathered}\end{equation}
\end{thm}
\begin{proof}
    Let $\ufk$ (respectively $\ukf$) denote the full subcategory of $\fk$ (resp. $\kf$) whose objects are the pairs $(A,0)$ with $A\in\Alg$. To alleviate notation, we write $A$ for an object $(A,0)$ either of $\ufk$, $\ukf$ or $\uks$. To define $\tF$ on objects, put $\tF(A)=F(A)$ for every $A\in\Alg$. 

    We will start by extending $F$ to $\ufk$ --- that is, we will define a functor $\hF$ making the following triangle commute:
    \begin{equation}\label{eq:ufkPU}\begin{gathered}\xymatrix@R=2em{\Alg\ar[r]^-{F}\ar@/_1pc/[dr]_-{j} & \cC\\
    & \ufk\ar@{-->}[u]_-{\hF}}\end{gathered}\end{equation}
    To define $\hF$ on objects, set $\hF(A):=F(A)$ for every algebra $A$.
    To define $\hF$ on morphisms, we must define a function
    \begin{equation}\label{eq:coliUPfk}\hF:\hom_\fk(A,B)=\colim_v [J^vA, B^{S^v}_\bullet]\to \hom_{\cC}(F(A), F(B))\end{equation}
    for every pair of algebras $(A,B)$. Fix the pair $(A,B)$. For every $v\geq 0$, let $\hF_v:[J^vA, B^{S^v}_\bullet]\to \hom_\cC(F(A), F(B))$ be the function that sends the class of $f:J^vA\to B^{S^v}_r$ to the following composite:
    \[\xymatrix{F(A)\ar[r]^-{(\alpha^{0,v}_A)^{-1}}_-{\cong} & F(\Sigma^vJ^vA)\ar[r]^-{f_*} & F(\Sigma^v B^{S^v}_r)\ar[r]^-{\gamma_*^{-1}}_-{\cong} & F(\Sigma^v B^{S^v})\ar[r]^-{\beta^{0,v}_B}_-{\cong} & F(B)}\]
    Here, $\gamma_*$ is the isomorphism  induced by the last vertex map; see Lemma \ref{lem:Fequivs} \eqref{item:FequivsLVM}. Note that $\hF_v$ is well defined on the set of homotopy classes since the functor $F(\Sigma^v-)$ is homotopy invariant. Let us show that the $\hF_v$ are compatible with the transition morphisms of the colimit \eqref{eq:coliUPfk}. The transition morphism of \eqref{eq:coliUPfk} sends the homotopy class of $f$ to the homotopy class of $\mu^{v,1}_B\circ \rho_{B^{S^v}}\circ J(f)$. Consider the following diagram in $\cC$:
    \[\xymatrix@C=3em{
    F(A)\ar[r]^-{(\alpha^{0,v+1}_A)^{-1}}_-{\cong}\ar@{}[dr]_(0.2){(1)} & F(\Sigma^{v+1}J^{v+1}A)\ar[r]^{J(f)_*}\ar@{}[d]^(.6){(2)} & F(\Sigma^{v+1}J(B^{S^v}_r))\ar[r]^-{(\rho_{B^{S^v}})_*}\ar[d]^-{\gamma_*^{-1}}_-{\cong} & F(\Sigma^{v+1}(B^{S^v}_r)^{S^1})\ar[d]^-{(\mu^{v,1}_B)_*} \\
    \save []+<1em,-2em>*+{F(\Sigma^v J^vA)}="A" \restore
    & &
    F(\Sigma^{v+1}J(B^{S^v}))\ar[d]^(.3){(\rho_{B^{S^v}})_*}\ar@{{}{ }{}}@/_3pc/[dd]|(.7){(3)}\ar@/_6pc/[dd]^-{\alpha^{v,1}_{B^{S^v}}}\ar@{}[ur]^{(4)}\ar@{}[dr]^(.7){(5)}
    & F(\Sigma^{v+1}B^{S^{v+1}}_r)\ar[d]^-{\gamma_*^{-1}}_-{\cong} \\
    & &
    F(\Sigma^{v+1}(B^{S^v})^{S^1})\ar@{}[dr]^(.7){(6)}
    & F(\Sigma^{v+1}B^{S^{v+1}})\ar[d]^-{\beta^{0,v+1}_B}_-{\cong} \\
    \save []+<1em,0.05em>*+{F(\Sigma^nB^{S^n}_r)}="B" \restore
    & & F(\Sigma^v B^{S^v})\ar[r]^-{\beta^{0,v}_B}_-{\cong} & F(B)
    \ar"3,3";"4,3"^-{\beta^{v,1}_{B^{S^v}}}
    \ar"3,3";"3,4"_-{(\mu^{v,1}_B)_*}
    \ar"1,4";"3,3"^(.55){\gamma_*^{-1}}_-{\cong}
    \ar"1,1";"A"+<-1em,0.9em>^-{(\alpha^{0,v}_A)^{-1}}_-{\cong}
    \ar"A"+<-1em,-0.8em>;"B"+<-1em,0.9em>^-{f_*}
    \ar"B"+<2.6em,-0.05em>;"4,3"^-{\gamma_*^{-1}}_-{\cong}
    \ar"1,2";"A"^-{\alpha^{v,1}_{J^vA}}
    \ar"1,3";"B"_-{\alpha^{v,1}_{B^{S^v}}}
    \ar@{}"B";"3,3"^(.4){(2)}
    }\]
    The triangle $(1)$ commutes by Lemma \ref{lem:alphapq}. The squares labelled $(2)$ commute by naturality of $\alpha^{v,1}_?$. The triangle $(3)$ commutes by definition of $\beta$. The square $(4)$ commutes by naturality of $\rho_?$. The square $(5)$ commutes since $\mu$ is compatible with the morphisms induced by the last vertex map. The square $(6)$ commutes by Lemma \ref{lem:betapq}. It follows that the outer square commutes. Composing the vertical morphisms in the left column followed by the horizontal morphisms in the bottom row we get $\hF_v([f])$. Composing the horizontal morphisms in the top row followed by the vertical morphisms in the right column we get $\hF_{v+1}([\mu^{v,1}_B\circ \rho_{B^{S^v}}\circ J(f)])$. This proves the desired compatibility and shows that the morphism \eqref{eq:coliUPfk} is indeed well defined. Let us now prove that it is compatible with composition. Let $f:J^pA\to B^{S^p}_r$ and $g:J^qB\to C^{S^q}_s$ be algebra homomorphisms. The composite $\langle g\rangle\circ\langle f\rangle$ in $\fk$ is defined as $\langle f\star g\rangle$, where $f\star g$ is the following composite in $[\Alg]^\ind$; see \cite{loopstho}*{Def. 3.4}.
    \[\xymatrix@C=4em{
    J^{p+q}A\ar[r]^-{J^q(f)} & J^q(B^{S^p}_\bullet)\ar[r]^-{(-1)^{qp}\kappa^{q,p}_B} & (J^q B)^{S^p}_\bullet\ar[r]^-{g^{S^p}} & (C^{S^q}_\bullet)^{S^p}_\bullet\ar[r]^-{\mu^{q,p}_C} & C^{S^{q+p}}_\bullet
    }\]
    Note that the sign makes sense since $[J^q(B^{S^p}_\bullet), (J^qB)^{S^p}_\bullet]$ is an abelian group.
    Consider the following commutative diagram in $\cC$:
    \[\xymatrix@C=8.5em{
    F(\Sigma^p J^pA)\ar[d]_-{f_*} & F(A)\ar[l]_-{(\alpha^{0,p}_A)^{-1}}^-{\cong}\ar[r]^-{(\alpha^{0,p+q}_A)^{-1}}_-{\cong} & F(\Sigma^{p+q}J^{p+q}A)\ar@/^2pc/[ll]_-{\alpha^{p,q}_{J^pA}}\ar[d]^-{(J^q(f))_*}\\
    F(\Sigma^p B^{S^p}_r) & & F(\Sigma^{p+q}J^q(B^{S^p}_r))\ar[ll]^-{\alpha^{p,q}_{B^{S^p}}}\ar[d]^-{(-1)^{qp}(\kappa^{q,p}_B)_*}\\
    \save []+<0em,1em>*+{F(\Sigma^pB^{S^p})}="L1" \restore
    \save []+<0em,-2em>*+{F(B)}="L2" \restore
    &
    \save []+<0em,1em>*+{F(\Sigma^{p+q}J^q(B^{S^p}))}="C1" \restore
    \save []+<0em,-2em>*+{F(\Sigma^{p+q}(J^qB)^{S^p})}="C2" \restore
    & F(\Sigma^{p+q}(J^qB)^{S^p}_r)\ar[d]^-{(g^{S^p})_*} \\
    \save []+<0em,-1em>*+{F(\Sigma^qJ^qB)}="L3" \restore
    &
    \save []+<0em,-1em>*+{F(\Sigma^{p+q}(C^{S^q}_s)^{S^p})}="C3" \restore
    & F(\Sigma^{p+q}(C^{S^q}_s)^{S^p}_r)\ar[d]^-{(\mu^{q,p}_C)_*} \\
    F(\Sigma^qC^{S^q}_s)\ar[d]_-{\gamma_*^{-1}}^{\cong} &
    \save []+<-4em,-1em>*+{F(\Sigma^{p+q}(C^{S^q})^{S^p})}="A" \restore
    \save []+<5.5em,0.05em>*+{F(\Sigma^{p+q}C^{S^{p+q}}_s)}="B" \restore
    & F(\Sigma^{p+q}C^{S^{p+q}}_{r+s})\ar[d]_-{\cong}^-{\gamma_*^{-1}} \\
    F(\Sigma^qC^{S^q})\ar[r]_-{\beta^{0,q}_C} & F(C) & F(\Sigma^{p+q}C^{S^{p+q}})\ar[l]^-{\beta^{0,p+q}_C}
    \ar"A";"6,1"_-{\beta^{q,p}_{C^{S^q}}}
    \ar"A";"6,3"^-{(\mu^{q,p}_C)_*}
    \ar"B";"6,3"^(.6){\gamma_*^{-1}}_(.5){\cong}
    \ar"C1";"L1"_-{\alpha^{p,q}_{B^{S^p}}}
    \ar"2,3";"C1"+<4em,0em>_-{\gamma_*^{-1}}^-{\cong}
    \ar"2,1";"L1"_-{\gamma_*^{-1}}^{\cong}
    \ar"L1";"L2"_-{\beta^{0,p}_B}
    \ar"L2";"L3"_-{(\alpha^{0,q}_B)^{-1}}^{\cong}
    \ar"L3";"5,1"_-{g_*}
    \ar"C1";"C2"_-{(-1)^{qp}(\kappa^{q,p}_B)_*}
    \ar"C2"+<-4em,-0.3em>;"L3"+<2.6em,0.2em>_-{\beta^{q,p}_{J^qB}}
    \ar"C2";"C3"_-{g_*}
    \ar"3,3"+<-4em,0em>;"C2"+<4em,0em>_-{\gamma_*^{-1}}^-{\cong}
    \ar"C3"+<-4em,-0.3em>;"5,1"+<2.5em,0.2em>_-{\beta^{q,p}_{C^{S^q}}}
    \ar"C3";"A"_-{\gamma_*^{-1}}^-{\cong}
    \ar"C3";"B"_(.4){(\mu^{q,p}_C)_*}
    \ar"4,3"+<-4em,0em>;"C3"+<3.9em,0em>_-{\gamma_*^{-1}}^-{\cong}
    \ar"5,3";"B"_-{\gamma_*^{-1}}^-{\cong}
    \ar@{}"2,1";"1,2"|(.75){(a)}
    \ar@{}"2,1";"1,2"|(.35){(b)}
    \ar@{}"1,1";"C1"_(.67){(b)}
    \ar@{}"L2";"C1"|-{(c)}
    \ar@{}"L2";"C3"_(.6){(d)}
    \ar@{}"L3";"A"^(.77){(d)}
    \ar@{}"A";"6,2"|-{(e)}
    }\]
    The triangle $(a)$ commutes by Lemma \ref{lem:alphapq}. The squares marked with $(b)$ commute by naturality of $\alpha^{p,q}_?$. The pentagon $(c)$ commutes by Lemma \ref{lem:kappaAB}. The squares marked with $(d)$ commute by naturality of $\beta^{q,p}_?$. The square $(e)$ commutes by Lemma \ref{lem:betapq}. The rest of the diagram clearly commutes. Starting at $F(A)$, if we first apply $(\alpha^{0,p}_A)^{-1}$, then the vertical morphisms on the left column and finally $\beta^{0,q}_C$, we get $\hF(\langle g\rangle)\circ \hF(\langle f\rangle)$. Again at $F(A)$, if we  first apply $(\alpha^{0,p+q}_A)^{-1}$, then the vertical morphisms on the right column and finally $\beta^{0,p+q}_C$, we get $\hF(\langle g\star f\rangle)$. This shows that the definitions above indeed define a functor $\hF$ making \eqref{eq:ufkPU} commute.

    Now we would like to extend $\hF$ to $\ukf$ --- that is, we want to define a functor $\vF$ making the following diagram commute:
    \begin{equation}\label{eq:ufkfPU}\begin{gathered}\xymatrix{
    \Alg\ar[rr]^-{F}\ar@/_1pc/[dr]_-{j} & & \cC \\
    & \ufk\ar[r]_-{t_f}\ar[ur]^-{\hF} & \ukf\ar@{-->}[u]_-{\vF}
    }\end{gathered}\end{equation}
    Recall from \cite{tesisema}*{Def. 5.1.3} that we have
    \[\hom_{\ukf}(A,B)=\colim_k\hom_{\ufk}(A, M_kB)\]
    where the transition morphisms are induced by the upper-left corner inclusions. If $f:A\to M_kB$ and $g:B\to M_lC$ are morphisms in $\ufk$, the composition of the corresponding morphisms in $\ukf$ is represented by the composite:
    \[\xymatrix{A\ar[r]^-{f} & M_kB\ar[r]^-{M_k(g)} & M_kM_lC\overset{\theta}{\cong}M_{kl}C}\]
    Here, $\theta$ stands for any bijection $\{1,\dots,k\}\times\{1,\dots, l\}\cong\{1,\dots,kl\}$.
    To define $\vF$ on objects, set $\vF(A):=F(A)$ for every algebra $A$. Let
    \[\vF:\hom_{\ukf}(A,B)\to\hom_{\cC}(F(A),F(B))\]
    be the function that sends $f\in\hom_{\ufk}(A,M_kB)$ to the composite
    \[\xymatrix{
    F(A)\ar[r]^-{\hF(f)} & F(M_kB)\ar[r]^-{F(\iota)^{-1}}_-{\cong} & F(B),
    }\]
    where $\iota:B\to M_kB$ is the upper-left corner inclusion. It is easily verified that these definitions indeed give rise to a functor $\vF$ making \eqref{eq:ufkfPU} commute. To prove the compatibility with the product, let $f:A\to M_kB$ and $g:B\to M_lC$ be morphisms in $\ufk$ and consider the following commutative diagram in $\ufk$:
    \[\xymatrix{A\ar[r]^-{f} & M_kB\ar[r]^-{M_k(g)} & M_kM_lC\ar[r]^-{\theta}_-{\cong} & M_{kl}C & C\ar[l]_-{\iota}\ar@/^1pc/[dll]_-{\iota} \\
    & B\ar[u]_-{\iota}\ar[r]^-{g} & M_lC\ar[u]^-{\iota}    
    }\]
    Upon applying $\hF$, the vertical morphisms become invertible since $\hF(\iota)=F(\iota)$. The resulting diagram in $\cC$ shows that $\vF$ preserves composition.

    The final step is to extend $\vF$ to $\uks$. Recall from \cite{tesisema}*{Def. 5.2.7} that we have:
    \[\hom_{\uks}(A,B)=\hom_{\ukf}(A, M_\infty B)\]
    Put $\tF(A):=F(A)$ for every algebra $A$, and let
    \[\tF:\hom_{\uks}(A,B)\to \hom_{\cC}(F(A),F(B))\]
    be the function that sends $f\in\hom_{\ukf}(A, M_\infty B)$ to the composite
    \[\xymatrix{
    F(A)\ar[r]^-{\vF(f)} & F(M_\infty B)\ar[r]^-{F(\iota)^{-1}}_-{\cong} & F(B)
    }\]
    where $\iota:B\to M_\infty B$ is the upper-left corner inclusion. It is straightforward to verify that this defines a functor $\tF$ making \eqref{eq:UPufk} commute.

    It remains to prove the uniqueness of $\tF$. Suppose that $\tF_1,\tF_2:\uks\to\cC$ are two functors making \eqref{eq:UPufk} commute. Both functors are equal on objects since they both coincide with $F$ on objects. Let $f:J^vA\to M_kM_\infty B^{S^v}_r$ be an algebra homomorphism representing a morphism $\langle f\rangle:A\to B$ in $\uks$. By Lemma \ref{lem:zzagkk}, $\tF_i(\langle f \rangle)$ equals the following composite in $\cC$:
    \[\xymatrix@C=4em{
    F(A)\ar[r]^-{\tF_i(\alpha^{0,v}_A)} & F(\Sigma^v J^vA)\ar[r]^-{\tF_i(\Sigma^v \tilde{f})} & F(\Sigma^v B^{S^v})\ar[r]_-{\cong}^-{\tF_i(\beta^{0,v}_B)^{-1}} & F(B)
    }\]
    We claim that $\tF_1(\Sigma^v\tilde{f})=\tF_2(\Sigma^v\tilde{f})$. Indeed, $\Sigma^v \tilde{f}$ is a zig-zag in $\ks$ of morphisms in the image of $j_s$ --- obtained upon applying $\Sigma^v(-)$ to \eqref{eq:tildef}. The functors $\tF_1$ and $\tF_2$ coincide on the image of $j_s$ since we have $\tF_1\circ j_s = F = \tF_2\circ j_s$. With the same argument we can show that $\tF_1(\alpha^{0,v}_A)=\tF_2(\alpha^{0,v}_A)$ and $\tF_1(\beta^{0,v}_B)=\tF_2(\beta^{0,v}_B)$ --- it is clear from their definitions that $\alpha$ and $\beta$ are zig-zags of morphisms in the image of $j_s$. It follows that $\tF_1(\langle f\rangle )=\tF_2(\langle f\rangle )$. This finishes the proof.
\end{proof}

As a corollary we get the following result.

\begin{thm}[cf. \cite{ln}*{Theorem 2.5}]\label{kkuniprop}
    The functor $j_s:\Alg\to\uks$ exhibits $\uks$ as the localization of $\Alg$ at the set of $\kk$-equivalences. More precisely, letting $W_{kk}$ be the set of $\kk$-equivalences in $\Alg$, the functor $j_s$ induces an equivalence $\Alg[W_{kk}^{-1}]\cong \uks$.
\end{thm}
\begin{proof}
    Let $l:\Alg\to\Alg[W_{kk}^{-1}]$ be the localization functor. Since $(\Alg, \Fib, W_{kk})$ is a category of fibrant objects by Proposition \ref{prop:algcof}, it follows from Lemma \ref{lem:cofproducts} that $\Alg[W_{kk}^{-1}]$ has finite products and that $l$ preserves them. Since $l(W_H\cup W_S\cup W_E)\subseteq \Iso(\Alg[W_{kk}^{-1}])$, by Theorem \ref{thm:kklocposta} there exists a unique functor $\phi:\uks\to \Alg[W_{kk}^{-1}]$ such that $l=\phi\circ j_s$. Since $j_s(W)\subseteq \Iso(\uks)$, by the universal property of $\Alg[W_{kk}^{-1}]$ there exists a unique functor $\psi:\Alg[W_{kk}^{-1}]\to\uks$ such that $j_s = \psi\circ l$. It follows from uniqueness that $\phi$ and $\psi$ are mutually inverses.
\end{proof}

\section{A stable infinity category realizing \texorpdfstring{$\kk$}{kk}}

The triangulated categories that arise naturally in mathematics are usually (but not always \cite{muro}) the homotopy categories of stable $\infty$-categories. In this section we show that this is the case for $\kk$, following what was done in \cite{ln} for Kasparov's $\KK$-theory. 

\subsection{Main theorem}
We will prove that the triangulated category $\kk$ arises as the homotopy category of a stable $\infty$-category $\kk_\infty$. We use the language of $\infty$-categories as developed in \cite{HT}. We write $\ho(\cC)$ for the homotopy category of an $\infty$-category $\cC$. Ordinary categories will be considered as $\infty$-categories using the nerve, though we will not write the nerve explicitely.

We start by recalling the Dwyer-Kan localization for $\infty$-categories \cite{HA}*{1.3.4.1 and 1.3.4.2}. Let $\cC$ be an $\infty$-category and let $W$ be a collection of morphisms in $\cC$. Then there exists an $\infty$-category $\cC[W^{-1}]$ endowed with a functor $L:\cC\to \cC[W^{-1}]$ such that, for any $\infty$-category $\cD$, $L$ induces an equivalence
\begin{equation}\label{eq:locequiv}\Fun(\cC[W^{-1}],\cD)\overset{\simeq}{\to} \Fun^W(\cC,\cD)\end{equation}
where $\Fun^W(\cC,\cD)$ is the full subcategory of $\Fun(\cC,\cD)$ of those functors that send morphisms in $W$ to equivalences in $\cD$. This universal property characterizes $\cC[W^{-1}]$ up to equivalence of $\infty$-categories.


Let $W_\kk$ be the set of $\kk$-equivalences in $\Alg$ and let
\begin{equation}\label{eq:localization}j_\infty:\Alg\to \kk_\infty:=\Alg[W_\kk^{-1}]\end{equation}
be the Dwyer-Kan localization at $W_\kk$. By the universal property of $j_\infty$ we have an equivalence
\[\Fun(\kk_\infty, \cC)\overset{\simeq}{\longrightarrow} \Fun^{W_\kk}(\Alg, \cC)\]
for every $\infty$-category $\cC$, induced by precomposing with $j_\infty$. It follows from this that there exists a functor $h:\kk_\infty\to \kk$ making the following triangle commute:
\[\xymatrix{\Alg\ar@/_1pc/[dr]_-{j_\infty}\ar[r]^-{j} & \kk \\
& \kk_\infty\ar[u]_-{h}}\]

We will prove the following result:
\begin{thm}\label{thm:main}
    The $\infty$-category $\kk_\infty$ is a stable infinity category and the functor $\ho(h):\ho(\kk_\infty)\to\ho(kk)$ is an equivalence of triangulated categories.
\end{thm}
The proof of this result will be done in several steps.

Recall that a stable $\infty$-category is a pointed $\infty$-category that admits finite limits and such that the loop functor is an equivalence \cite{HA}*{Cor 1.4.2.27}.
In order to prove the existence of finite limits in $\kk_\infty$, we will calculate the localization \eqref{eq:localization} using a category of fibrant objects structure on $\Alg$ described in Section \ref{seccof}.

\begin{lem}\label{lem:kkinflimits}
    \begin{enumerate}
        \item[]
        \item\label{item:kkpointed} The $\infty$-category $\kk_\infty$ is pointed and the algebra $0$ is a zero object.
        \item\label{item:finlim} The $\infty$-category $\kk_\infty$ has finite limits.
        \item\label{item:ex} The functor $j_\infty:\Alg\to\kk_\infty$ sends Milnor squares to pullback squares.
    \end{enumerate}
\end{lem}
\begin{proof}
    The assertion \eqref{item:kkpointed} follows from \cite{cis}*{Remark 7.1.15} since $0$ is both final and initial in $\Alg$.
    Let $\Fib$ be the set of those surjective morphisms in $\Alg$ that have a linear section and let $W_\kk$ be the set of $\kk$-equivalences in $\Alg$. Then $(\Alg, \Fib, W_\kk)$ is an $\infty$-category of fibrant objects in the sense of \cite{cis}*{Def. 7.5.7} --- see also \cite{cis}*{Def. 7.4.12}. Indeed, this follows immediately from Proposition \ref{prop:algcof} --- property (ii) in \cite{cis}*{Def. 7.4.12} holds by \cite{brown}*{Factorization Lemma}. The assertions \eqref{item:finlim} and \eqref{item:ex} now follow from \cite{cis}*{Proposition 7.5.6}.
\end{proof}

Let $\cC$ be a pointed $\infty$-category with finite limits. We proceed to recall from \cite{HA}*{1.1.2} the construction of the loop functor $\Omega_\cC:\cC\to \cC$. Let $\MOm$ be the full subcategory of $\Fun(\Delta^1\times\Delta^1,\cC)$ of those diagrams
\[\xymatrix{Y\ar[d]\ar[r] & 0\ar[d] \\ 0'\ar[r] & X}\]
that are pullback squares and where $0$ and $0'$ are zero objects of $\cC$. If $\cC$ admits finite limits, it can be shown that the evaluation at $(1,1)$ induces a trivial fibration $\MOm\to \cC$. Let $s:\cC\to\MOm$ be a section of this trivial fibration and let $e:\MOm\to \cC$ be the functor given by evaluation at $(0,0)$. Then $\Omega_\cC$ is defined as the composite $e\circ s$. For $\cC=\kk_\infty$ we can give a description of this loop functor using loop algebras.

\begin{lem}[cf. \cite{bel}*{Lemma 2.6}]\label{lem:omegas}
    \begin{enumerate}
        \item[]
        \item\label{item:omegacat} The functor $\Omega:\Alg\to\Alg$, $A\mapsto \Omega A$, uniquely descends to an equivalence of additive categories $\Omega:\kk\to\kk$ completing the square:
        \[\xymatrix{\Alg\ar[d]_-{j}\ar[r]^-{\Omega} & \Alg\ar[d]^-{j} \\
        \kk\ar@{-->}[r]^-{\Omega} & \kk}\]
        \item\label{item:omegainf} The functor $\Omega:\Alg\to\Alg$, $A\mapsto \Omega A$, essentially uniquely descends to a functor $\Omega:\kk_\infty\to\kk_\infty$ completing the square:
        \begin{equation}\label{diagram:omegainf}\begin{gathered}\xymatrix{\Alg\ar[d]_-{j_\infty}\ar[r]^-{\Omega} & \Alg\ar[d]^-{j_\infty} \\
        \kk_\infty\ar@{-->}[r]^-{\Omega} & \kk_\infty}\end{gathered}\end{equation}
        \item\label{item:loopsagree} The loop functor $\Omega_{\kk_\infty}:\kk_\infty\to\kk_\infty$ is naturally equivalent to the functor $\Omega:\kk_\infty\to\kk_\infty$ of \eqref{diagram:omegainf}.
    \end{enumerate}
\end{lem}
\begin{proof}
    The assertion \eqref{item:omegacat} is well known. Indeed, the functor $\Omega:\kk\to\kk$ is the desuspension in the triangulated category $\kk$; see \cite{cortho}*{Sections 6.3 and 6.4}.
    
    The functor $\Omega:\Alg\to\Alg$ preserves $\kk$-equivalences by $\eqref{item:omegacat}$. Thus, the composite $j_\infty\circ \Omega:\Alg\to\kk_\infty$ sends $\kk$-equivalences to equivalences.
    Then \eqref{item:omegainf} follows from the universal property of $\kk_\infty$ being a Dwyer-Kan localization.
    
    Let us prove \eqref{item:loopsagree}.
    For $A\in\Alg$ we have a Milnor square
    \begin{equation}\label{eq:MilnorLoop}\begin{gathered}\xymatrix{\Omega A\ar[d]\ar[r] & PA\ar[d]^-{\ev_1} \\
    0\ar[r] & A}\end{gathered}\end{equation}
    where $PA=tA[t]$ and $\Omega A=(t^2-t)A[t]$ --- note that $a\mapsto at$ is a linear section of $\ev_1$. Since $PA$ is contractible, it is a zero object in $\kk_\infty$. Thus, upon applying $j_\infty$ to \eqref{eq:MilnorLoop} we get a pullback square that is equivalent to the one defining $\Omega_{\kk_\infty}$.
    This finishes the proof.
\end{proof}

\begin{lem}[cf. \cite{ln}*{Proposition 3.3}]\label{lem:main}
\begin{enumerate}
        \item[]
        \item\label{item:equiv} The functor $\ho(h): \ho(\kk_\infty)\to \ho(\kk)$ is an equivalence of ordinary categories.
        \item\label{item:stable} The $\infty$-category $\kk_\infty$ is stable.
    \end{enumerate}
\end{lem}
\begin{proof}
    Let us prove \eqref{item:equiv}. Let $W$ denote the set of $\kk$-equivalences in $\Alg$. If $\cC$ is an infinity category, write $\gamma_\cC:\cC\to\ho(\cC)$ for the canonical functor. Since forming localizations is compatible with going over to homotopy categories, we have a commutative square
    \[\xymatrix{\Alg\ar[d]_-{\gamma_\Alg}\ar[r]^-{j_\infty} & \kk_\infty\ar[d]^-{\gamma_{\kk_\infty}} \\
    \ho(\Alg)\ar[r]^-{\ho(j_\infty)} & \ho(\kk_\infty)}\]
    where $\ho(j_\infty)$ presents $\ho(\kk_\infty)$ as the localization of $\ho(\Alg)$ at the set $\ho(W)$ in the sense of ordinary categories. Since $\Alg$ is an ordinary category, then $\gamma_\Alg$ is an equivalence and $\ho(j_\infty)\circ\gamma_\Alg$
    presents its $\ho(\kk_\infty)$ as the localization of $\Alg$ at $W$.
    Now consider the following commutative diagram of ordinary categories:
    \[\xymatrix{\Alg\ar[d]_-{j}\ar[r]^-{\gamma_\Alg} & \ho(\Alg)\ar[r]^-{\ho(j_\infty)}\ar[dr]_-{\ho(j)} &  \ho(\kk_\infty)\ar[d]^-{\ho(h)} \\
    \kk\ar[rr]_-{\gamma_\kk} & & \ho(\kk)}\]
    One one hand, we have seen that $\ho(j_\infty)\circ\gamma_\Alg$ presents $\ho(\kk_\infty)$ as the localization of $\Alg$ at $W$. On the other, the functor $\gamma_\kk\circ j$ presents $\ho(\kk)$ as the localization of $\Alg$ at $W$. Indeed, the latter follows from Theorem \ref{kkuniprop} and the fact that $\kk$ is an ordinary category --- and thus $\gamma_\kk$ is an equivalence. Since both composites share the same universal property, it follows that $\ho(h)$ is an equivalence.
    
    To prove \eqref{item:stable} we have to show that the loop functor on $\kk_\infty$, $\Omega_{\kk_\infty}$, is an equivalence. By \cite{ln}*{Lem. 3.4 (2)} it suffices to prove that $\Omega_{\kk_\infty}$ induces an equivalence on $\ho(\kk_\infty)$. By Lemma \ref{lem:omegas} \eqref{item:loopsagree}, it is enough to show that $\Omega:\kk_\infty\to\kk_\infty$ induces an equivalence on $\ho(\kk_\infty)$. In view of \eqref{item:equiv}, the latter is identified with $\Omega:\kk\to\kk$, that is an equivalence by Lemma \ref{lem:omegas} \eqref{item:omegacat}. 
\end{proof}

\begin{lem}[cf. \cite{bel}*{Lemma 2.17}]\label{lem:triangleshokk}
    \begin{enumerate}
        \item[]
        \item\label{item:morkkinf} Every morphism in $\kk_\infty$ is equivalent to a morphism $j_\infty(f)$ for some morphism $f$ in $\Alg$.
        \item\label{item:kkinfprod} Every product in $\kk_\infty$ is equivalent to the image under $j_\infty$ of a product in $\Alg$.
        \item\label{item:triangleshokk} Every distinguished triangle in $\ho(\kk_\infty)$ is isomorphic to one of the form
        \[\xymatrix@C=3em{\Omega B\ar[r]^-{[j_\infty(\iota_f)]} & P_f\ar[r]^-{[j_\infty(\pi_f)]} & A\ar[r]^-{[j_\infty(f)]} & B}\]
        for some algebra homomorphism $f:A\to B$; see \eqref{eq:pathseq} for the definitions of $\iota_f$ and $\pi_f$. Here, we are identifying $\Omega_{\kk_\infty}(j_\infty(B))\simeq j_\infty(\Omega B)$ as explained in Lemma \ref{lem:omegas}.
    \end{enumerate}
\end{lem}
\begin{proof}
    To prove \eqref{item:morkkinf}, let $g$ be a morphism in $\kk_\infty$. Since $\kk_\infty$ is the $\infty$-category associated to the $\infty$-category of fibrant objects $(\Alg, \Fib, W_{\kk})$, the class $[g]$ in $\ho(\kk_\infty)$ equals $[j_\infty(f)]\circ[j_\infty(s)]^{-1}$ for some morphisms $f$ and $s$ in $\Alg$, with $s$ a $\kk$-equivalence --- see, for example, the proof of \cite{cis}*{Thm. 7.5.6}. Then $j_\infty(f)$ is a composition of the morphisms $j_\infty(s)$ and $g$ in $\kk_\infty$ and therefore the morphisms $g$ and $j_\infty(f)$ are equivalent.

    Let us prove \eqref{item:kkinfprod}. First note that any object of $\kk_\infty$ is equivalent to one in the image of $j_\infty$. Let $j_\infty(A)$ and $j_\infty(B)$ be two objects of $\kk_\infty$. Since the square of algebras
    \[\xymatrix{A\times B\ar[d]\ar[r] & A\ar[d] \\ B\ar[r] & 0}\]
    is a Milnor square, it becomes a pullback upon applying $j_\infty$ by Lemma \ref{lem:kkinflimits} \eqref{item:ex}. This implies that $j_\infty(A)\times j_\infty(B)\simeq j_\infty(A\times B)$ in $\kk_\infty$.

    Let us prove \eqref{item:triangleshokk}. Since $\ho(\kk_\infty)$ is a triangulated category, every morphism in $\ho(\kk_\infty)$ can be extended to a distinguished triangle that is uniquely determined up to isomorphism, and every distinguished triangle arises in this way. By \eqref{item:morkkinf}, it suffices to consider those distinguished triangles that extend a morphism of the form $[j_\infty(f)]$ for some morphism $f$ in $\Alg$. Let $f:A\to B$ be an algebra homomorphism and consider the following diagram of algebras:
    \[\xymatrix{\Omega B\ar[d]\ar[r]^-{\iota_f} & P_f\ar[d]^-{\pi_f}\ar[r] & PB\ar[d]^-{\ev_1} \\0\ar[r] & A\ar[r]^-{f} & B}\]
    Since the vertical morphisms belong to $\Fib$, both squares are Milnor squares and thus become pullbacks in $\kk_\infty$. The assertion \eqref{item:triangleshokk} now follows from (the dual of) \cite{HA}*{Def. 1.1.2.11}.
\end{proof}

\begin{prop}[cf. \cite{bel}*{Prop. 2.18}]
    The functor $\ho(h): \ho(\kk_\infty)\to \ho(\kk)$ is an equivalence of triangulated categories.
\end{prop}
\begin{proof}
    By Lemma \ref{lem:main} \eqref{item:equiv}, the functor $\ho(h)$ is an equivalence of ordinary categories. Since binary products in both $\ho(\kk_\infty)$ and $\kk$ are represented by the cartesian product of algebras, it follows that $\ho(h)$ preserves binary products and that it is an additive functor. By Lemma \ref{lem:omegas}, it follows that $\ho(h)$ is compatible with the desuspension functors. Finally, it follows from Sec. \ref{sec:kk} and Lemma \ref{lem:triangleshokk} \eqref{item:triangleshokk} that $\ho(h)$ sends distinguished triangles in $\ho(\kk_\infty)$ to distinguished triangles in $\kk$. This finishes the proof.
\end{proof}

This finishes the proof of Theorem \ref{thm:main}.

\subsection{Inverting polynomial homotopy equivalences}
A stable $\infty$-category version of $KK$-theory is constructed in \cite{bunke} by making a sequence of localizations of the category of $\CAlg$ to enforce the universal properties of $KK$.
An alternative approach to constructing $\kk_\infty$ would be to mimic the steps described in \cite{bunke}*{p.3} in our algebraic context. We will show that this can be done to some extent in our setting, though it is not clear for the authors how to go further on. The first of these steps consists of inverting homotopy equivalences.

\begin{defi}
    Let $L_h:\Alg\to\Alg_h$ be the Dwyer-Kan localization of $\Alg$ at the set $W_H$ of polynomial homotopy equivalences.
\end{defi}

Let $\Alg_\infty$ be the $\infty$-category that results from applying the homotopy coherent nerve to the enrichment of $\Alg$ over Kan complexes described in \ref{sec:enrichKan}. We have a functor $\Alg\to\Alg_\infty$ induced by the inclusions of the zero skeletons of the mapping spaces.

\begin{prop}[cf. \cite{bunke}*{Proposition 3.5}]\label{prop:alghcoh}
    The functor $\Alg\to\Alg_\infty$ presents the Dwyer-Kan localization of $\Alg$ at the set $W_H$ of polynomial homotopy equivalences.
\end{prop}
\begin{proof}
    Let $\cC$ be the opposite category of the simplicially enriched category of algebras and let $\cC'$ be the opposite category of the Kan complex enriched category of algebras. Eplicitely, $\ob(\cC)=\ob(\cC')=\ob(\Alg)$ and for algebras $A$ and $B$ we have:
    \begin{align*}\Map_\cC(A,B)&=\uhom(B,A)\\
    \Map_{\cC'}(A,B)&=\uHOM(B,A)\end{align*}
    Note that the underlying ordinary category of $\cC$ is $\cC_0=\Alg^\op$ and that $\cC'$ is a fibrant replacement of $\cC$ as a simplicially enriched category. Let $W$ be the collection of morphisms in $\cC$ corresponding to polynomial homotopy equivalences. Following the proof of \cite{bunke}*{Proposition 3.5}, we would like to invoke \cite{HA}*{Proposition 1.3.4.7} but this is not possible because the exponential law fails in this algebraic context; see \cite{cortho}*{Remark 3.1.4}. Nevertheless, we will show that the proof of \cite{HA}*{Proposition 1.3.4.7} carries on with some changes. More precisely, we will show that the canonical morphism $\theta:(N(\cC_0),W)\to (N(\cC'),W)$ is a weak equivalence of marked simplicial sets.
    
    Let $f:A\to B$ be a morphism in $\cC$ that belongs to $W$ --- that is, a zero simplex $f\in\Map_\cC(A,B)_0=\hom_\Alg(B,A)$ that is a polynomial homotopy equivalence. It is clear that if $g:A\to B$ is any morphism that belongs to the same connected component of $\Map_\cC(A,B)$ as $f$, then $g$ belongs to $W$. Indeed, unravelling the definitions, $f$ and $g$ are connected by an edge of $\Map_\cC(A,B)$ if and only if they are elementarily homotopic.

    Fix $n\geq 0$ and let $\cC_n$ be the ordinary category whose objects are those of $\cC$ and whose morphisms are given by:
    \[\hom_{\cC_n}(A,B)=\hom_\sSet(\Delta^n,\Map_\cC(A,B))=\hom_\Alg(B, A^{\Delta^n})\]
    Let $W_n$ be the collection of morphisms in $\cC_n$ corresponding to morphisms $\Delta^n\to\Map_\cC(A,B)$ that carry each vertex of $\Delta^n$ to an element of $W$. As explained in the proof of \cite{HA}*{Proposition 1.3.4.7}, we must show that the morphism $F:(N(\cC_0),W_0)\to (N(\cC_n),W_n)$ induced by $[n]\to[0]$ is a weak equivalence of marked simplicial sets. Let $G:(N(\cC_n),W_n)\to (N(\cC_0),W_0)$ be induced by the inclusion $[0]\to [n]$. Then $G\circ F=\id$. We will show that $F\circ G$ is homotopic to the identity of $\cC_n$. Define a functor $U:\cC_n\to\cC_n$ as follows:
    \begin{itemize}
        \item On objects, $U$ is given by $U(A)=A^{\Delta^1}$.
        \item On morphisms, we should define compatible functions:
        \[\xymatrix{\hom_{\cC_n}(A,B)=\hom_{\Alg}(B, A^{\Delta^n})\ar@<-4em>[d]_-{U} \\ \hom_{\cC_n}(U(A), U(B))=\hom_{\Alg}(B^{\Delta^1}, (A^{\Delta^1})^{\Delta^n})}\]
        Recall that we have isomorphisms:
        \[\iota_p:\ell^{\Delta^p}=\dfrac{\ell[t_0,\dots,t_p]}{\langle1-\sum t_i\rangle}\to \ell[s_1,\dots,s_p]\text{,\hspace{2em}}\iota_p(t_i)=s_i\text{\hspace{1em}for $i>0$.}\]
        Define $\rho:(\ell^{\Delta^n})^{\Delta^1}\to(\ell^{\Delta^n})^{\Delta^1}$ as the composite
        \[\xymatrix{(\ell^{\Delta^n})^{\Delta^1}\ar[r]^-{\iota_n\otimes\iota_1}_-{\cong} & \ell[s_1,\dots ,s_n][\sigma]\ar[r]^-{r} & \ell[s_1,\dots ,s_n][\sigma] & (\ell^{\Delta^n})^{\Delta^1}\ar[l]_-{\iota_n\otimes\iota_1}^-{\cong}}\]
        where $r$ is determined by $r(\sigma)=\sigma$ and $r(s_i)=\sigma s_i$ for $1\leq i\leq n$. Upon tensoring $\rho$ with $A$ we get an algebra homomorphism $\rho:(A^{\Delta^n})^{\Delta^1}\to(A^{\Delta^n})^{\Delta^1}$.
        Let $f\in\hom_{\cC_n}(A,B)$ be represented by $f:B\to A^{\Delta^n}$. Define $U(f)$ as the morphism represented by the composite:
        \[\xymatrix{B^{\Delta^1}\ar[r]^-{f^{\Delta^1}} & (A^{\Delta^n})^{\Delta^1}\ar[r]^-{\rho} & (A^{\Delta^n})^{\Delta^1}\cong (A^{\Delta^1})^{\Delta^n}}\]
    \end{itemize}
    It is straightforward but tedious to verify that the latter defines a functor $U$. For $i=0,1$, let $\const_i:[n]\to[1]$ denote the constant function taking the value $i$ and let $\tau_i\in\hom_{\cC_n}(A, U(A))=\hom_{\Alg}(A^{\Delta^1}, A^{\Delta^n})$ be represented by:
    \[(\const_i)^*:A^{\Delta^1}\to A^{\Delta^n}\]
    Another straightforward verification shows that the $\tau_i$ assemble into natural transformations $\tau_0:F\circ G\to U$ and $\tau_1:\id_{\cC_n}\to U$. Note that both $\tau_0$ and $\tau_1$ are given by morphisms in $W_n$ (any morphism $A^{\Delta^1}\to A$ induced by a morphism $\Delta^0\to\Delta^1$ is a polynomial homotopy equivalence). It follows that $\tau_0$ and $\tau_1$ determine the desired homotopy from $F\circ G$ to the identity of $\cC_n$.
\end{proof}

\begin{coro}[cf. \cite{bunke}*{Corollary 3.7}]\label{coro:MapAlgh}
    For any two algebras $A$ and $B$ we have a natural equivalence of spaces:
    \[\Map_{\Alg_h}(A,B)\simeq \uHOM(A,B)\]
\end{coro}

\begin{prop}[cf. \cite{bunke}*{Propostion 3.8}]\label{prop:AlghProdCoprod}
    The category $\Alg_h$ admits finite products and finite coproducts, and $L_h$ preserves them.
\end{prop}
\begin{proof}
It follows from Corollary \ref{coro:MapAlgh} and the isomorphisms of simplicial sets
    \[\uHOM(A, \textstyle\prod_{i\in I} B_i) \cong  \textstyle\prod_{i\in I}\uHOM(A, B_i)
    \]
    \[\uHOM(\textstyle\bigast_{i\in I}A_i, B)\cong \textstyle\prod_{i\in I}\uHOM(A_i,B)\]
for finite $I$.
\end{proof}

\begin{lem}[cf. \cite{bunke}*{Lemma 3.10}]\label{lem:AlghPointed}
    The category $\Alg_h$ is pointed and the functor $L_h:\Alg\to\Alg_h$ is reduced --- i.e.\ it preserves zero objects.
\end{lem}
\begin{proof}
    The zero algebra represents the initial and final object of $\Alg_h$ since
    \[\uHOM(0,A)\cong \Delta^0\cong \uHOM(A,0)\]
    for every algebra $A$.
\end{proof}

At this point, a difference arises between the topological and the algebraic contexts. In \cite{bunke}*{Def. 3.2}, Bunke defines $L_h:\CAlg\to\CAlg_h$ as the Dwyer-Kan localization of the category of $C^*$-algebras at the homotopy equivalences. Later on, he uses the category of fibrant objects structure on $(\CAlg, \FSch, W_h)$ to prove that $\CAlg_h$ has finite limits and that $L_h$ sends Schochet fibrant cartesian squares in $\CAlg$ to pullbacks in $\CAlg_h$; see \cite{bunke}*{Proposition 3.17}. This is a key point in Bunke's construction of the $\infty$-category version of $KK$-theory since the existence of finite limits in $\CAlg_h$ is carried over by formal arguments to the succesive steps that enforce the universal properties of $KK$. In our algebraic context, the authors do not know whether $\Alg_h$ has finite limits, though it is easily verified that the Milnor squares
\[\begin{gathered}\xymatrix{\Omega_i\ar[d]\ar[r] & P_i\ar[d]\\
0\ar[r] & \ell}\end{gathered}\qquad (i=0,1)\]
do not give pullbacks in $\Alg_h$ simultaneously. Indeed, if they did, they would induce a morphism of homotopy pullbacks of Kan complexes as in \eqref{eq:cuadradoHOM} and the morphism $\beta:\Omega_0\to\Omega_1$ would be an homotopy equivalence. We have seen in Lemma \ref{betano} that this is not the case if $\ell$ is a domain with $KH_0(\ell)\neq 0$.

\appendix
\section{Algebraic \texorpdfstring{$\kk$}{kk}-theory}

Let $A\in\Alg$. We proceed to recall some extensions associated to $A$.

\subsection{The cone extension}
Recall from \cite{cortho}*{Section 4.7} that we have an extension of rings
\[\xymatrix{M_\infty\ar[r] & \Gamma\ar[r] & \Sigma}\]
that splits as a sequence of abelian groups. For the rest of the paper, we fix a splitting $\tau:\Sigma\to \Gamma$. Upon tensoring with $A$, we get an extension of algebras:
\begin{equation}\label{eq:coneext}
    \xymatrix{\scrC_A: M_\infty A\ar[r] & \Gamma A\ar[r]^-{q_A} & \Sigma A}
\end{equation}
We will always consider $\tau_A=\tau\otimes A$ as a splitting for \eqref{eq:coneext}. We call \eqref{eq:coneext} the \emph{cone extension} of $A$. It is clear that an algebra homomorphism $A\to B$ induces a strong morphism of extensions $\scrC_A\to \scrC_B$. We write $x_A:J\Sigma A\to M_\infty A$ for the classifying map of $\scrC_A$.

\subsection{The universal extension}
Recall from \cite{loopstho}*{Sec. 2.22} that we have the \emph{universal extension} of $A$
\[\xymatrix{\scrU_A: JA\ar[r] & TA\ar[r]^-{\eta_A} & A}\]
with splitting $\sigma_A$. An algebra homomorphism $A\to B$ induces a strong morphism of extensions $\scrU_A\to \scrU_B$.

\subsection{The path extension}
Let $n,r\geq 0$. Recall from \cite{loopstho}*{Sec. 2.28} that we have the \emph{path extension} of $A$:
\begin{equation}\label{eq:pathext}
    \xymatrix{\scrP_{n,A}: A^{S^{n+1}}_r\ar[r] & P(n,A)_r\ar[r] & A^{S^n}_r}
\end{equation}
 An algebra homomorphism $A\to B$ induces a strong morphism of extensions $\scrP_{n,A}\to \scrP_{n,B}$. Moreover, the last vertex map induces strong morphisms of extensions $\scrP_{n,A,r}\to\scrP_{n, A, r+1}$. Explicit descriptions of these extensions for $n=0$ and $r=0,1$ are given in \cite{loopstho}*{Example 2.32}. We write $\rho_A:JA\to A^{S^1}_r$ for the classifying map of $\scrP_{0,A}$.

\subsection{Every morphism in \texorpdfstring{$\kk$}{kk} is a zig-zag of algebra homomorphisms}

\begin{lem}\label{lem:extdoble}
    Let $A\in\Alg$. Then the following square in $\fk$ anticommutes (i.e.\ one way equals minus the other way):
    \[\xymatrix@C=4em{
    (J\Sigma A, 0)\ar[r]^-{\langle J(\rho_{\Sigma A})\rangle}\ar[d]_-{(x_A)_*} & (J(\Sigma A^{S^1}), -1)\ar[d]^-{\left(x_{A^{S^1}}\right)_*} \\
    (M_\infty A, 0)\ar[r]^-{\langle \rho_{M_\infty A} \rangle} & (M_\infty A^{S^1}, -1)
    }\]
\end{lem}
\begin{proof}
    It is easily verified that the composite of the top morphism followed by the right morphism is represented by $x_{A^{S^1}_1}\circ J(\rho_{\Sigma A}):J^2\Sigma A\to M_\infty A^{S^1}_1$. Moreover, the composite of the left morphism followed by the bottom one is represented by
    $\rho_{M_\infty A}\circ J(x_A):J^2\Sigma A\to M_\infty A^{S^1}_1$.
    Thus, to prove that the square anticommutes it suffices to show that $x_{A^{S^1}_1}\circ J(\rho_{\Sigma A})$ and $\rho_{M_\infty A}\circ J(x_A)$ are mutually inverses in the group $[J^2\Sigma A, (M_\infty A)^{S^1}_\bullet]$.
    The idea is the same as in \cite{loopstho}*{Lemma 4.3}. Define
    \[I=\ker\left(\xymatrix{t(\Gamma A)[t]\ar[r]^-{\ev_1} & \Gamma A\ar[r]^{q_A} & \Sigma A}\right)\]
    and let $s:\Sigma A\to t(\Gamma A)[t]$ be given by $s(a)=\tau_A(a)t$ --- recall that $\tau_A$ is the section of the cone extension $\scrC_A$ \eqref{eq:coneext}. We have an extension:
    \[(\scrE, s):\xymatrix@C=4em{I\ar[r] & t(\Gamma A)[t]\ar[r]^-{q_A\circ \ev_1} & \Sigma A }\]
    Now put
    \[E=\left\{(u,v)\in t(\Gamma A)[t]\times t(M_\infty A)[t] \mid u(1)=v(1)\right\}\]
    and let $s':I\to E$ be given by $s'(u)=(u, u(1)t)$ --- here we use that $\ev_1:I\to\Gamma A$ factors through $M_\infty A$. We have an extension:
    \[(\scrE', s'):\xymatrix@C=4em{(t^2-t)(M_\infty A)[t]\ar[r]^-{(0, \inc)} & E\ar[r]^-{\pr_1} & I}\]

    Let $\chi:J\Sigma A\to I$ be the classifying map of $(\scrE, s)$. The following diagram shows a strong morphism of extensions $\scrU_{\Sigma A}\to \scrC_A$ extending the identity of $\Sigma A$:
    \[\xymatrix@C=4em{
    \scrU_{\Sigma A}\ar[d] & J\Sigma A\ar[d]_-{\chi}\ar[r] & T\Sigma A\ar[d]\ar[r] & \Sigma A\ar@{=}[d]\\
    (\scrE, s)\ar[d] & I\ar[d]_-{\ev_1}\ar[r] & t(\Gamma A)[t]\ar[d]^-{\ev_1}\ar[r] & \Sigma A\ar@{=}[d]\\
    \scrC_A & M_\infty A\ar[r] & \Gamma A\ar[r] & \Sigma A
    }\]
    It follows that $\ev_1\circ\chi=x_A$. Now let $\omega$ be the automorphism of $(M_\infty A)[t]$ given by $\omega(t)=1-t$ and consider the following strong morphism of extensions $\scrU_{J\Sigma A}\to\scrP_{0,M_\infty A}$ extending $x_A$:
    \[
    \xymatrix@C=3em{
    \scrU_{J\Sigma A}\ar[d] & J^2\Sigma A\ar[d]\ar[r] & TJ\Sigma A\ar[d]\ar[r] & J\Sigma A\ar[d]^-{\chi}\ar@/^2pc/[dd]^-{x_A}\\
    (\scrE',s')\ar[d] & (t^2-t)(M_\infty A)[t]\ar@{=}[d]\ar[r]^-{(0, \inc)} & E\ar[d]^-{\pr_2}\ar[r]^-{\pr_1} & I\ar[d]^-{\ev_1}\\
    \scrE''\ar[d] & (t^2-t)(M_\infty A)[t]\ar[d]_-{\omega}\ar[r] & t(M_\infty A)[t]\ar[d]^-{\omega}\ar[r]^-{\ev_1} & M_\infty A\ar@{=}[d]\\
    \scrP_{0,M_\infty A} & (t^2-t)(M_\infty A)[t]\ar[r] & (1-t)(M_\infty A)[t]\ar[r]^-{\ev_0} & M_\infty A
    }
    \]
    Since $\omega^{-1}=\omega$, it follows from \cite{loopstho}*{Proposition 2.26} that the classifying map of $\chi$ with respect to $(\scrE', s')$ equals the composite:
    \[\xymatrix{J^2\Sigma A\ar[r]^-{J(x_A)} & J(M_\infty A)\ar[r]^-{\rho_{M_\infty A}} & M_\infty A^{S^1}_0\ar[r]^-{\omega} & M_\infty A^{S^1}_0}\]
    Define $\theta:E\to (\Gamma A)^{S^1}_1=t(\Gamma A)[t]\tensor[_{\ev_1}]{\times}{_{\ev_1}} t(\Gamma A)[t]$ by $\theta(u,v)=(v,u)$ and consider the following morphisms of extensions:
    \begin{equation}\begin{gathered}\label{eq:notstr}
    \xymatrix@C=4em{
    \scrU_{J\Sigma A}\ar[d] & J^2\Sigma A\ar[d]_-{\omega\circ\rho_{M_\infty A}\circ J(x_A)}\ar[r] & TJ\Sigma A\ar[r]\ar[d] & J\Sigma A\ar[d]^-{\chi} \\
    \scrE'\ar[d] & (M_\infty A)^{S^1}_0\ar[r]\ar[d]_-{(\inc, 0)} & E\ar[d]^-{\theta}\ar[r] & I\ar[d]^-{(0, q_A)} \\
    (\scrC_A)^{S^1}_1 & (M_\infty A)^{S^1}_1\ar[r] & (\Gamma A)^{S^1}_1\ar[r] & (\Sigma A)^{S^1}_1
    }\end{gathered}\end{equation}
    Note that the morphism $\scrE'\to(\scrC_A)^{S^1}_1$ is not strong and that $(\scrC_A)^{S^1}_1\cong \scrC_{A^{S^1}_1}$. Put $\psi:=(0, q_A)\circ \chi:J\Sigma A\to (\Sigma A)^{S^1}_1$. It follows from \cite{loopstho}*{Proposition 2.26} applied to \eqref{eq:notstr} that the following diagram commutes in $[\Alg]$:
    \begin{equation}\label{eq:halg}\begin{gathered}\xymatrix{
    J^2\Sigma A\ar@{=}[dd]\ar[rrr]^-{J(\psi)} & & & J[(\Sigma A)^{S^1}_1]\ar@{=}[r] & J\Sigma (A^{S^1}_1)\ar[d]^-{x_{A^{S^1}_1}} \\
    & & & & M_\infty(A^{S^1}_1)\ar@{=}[d] \\
    J^2\Sigma A\ar[r]^-{J(x_A)} & J(M_\infty A)\ar[r]^-{\rho_{M_\infty A}} & (M_\infty A)^{S^1}_0\ar[r]^-{\omega} & (M_\infty A)^{S^1}_0\ar[r]^-{\gamma} & (M_\infty A)^{S^1}_1
    }\end{gathered}\end{equation}
    Here $\gamma=(\inc, 0): (t^2-t)(M_\infty A)[t]\to t(M_\infty A)[t]\tensor[_{\ev_1}]{\times}{_{\ev_1}}t(M_\infty A)[t]$ is the morphism induced by the last vertex map; see \cite{loopstho}*{Example 2.32}. We claim that $J(\psi)=J(\rho_{\Sigma A})$ in $[\Alg]$. By \cite{loopstho}*{Lemma 2.27} it suffices to show that $\psi=\rho_{\Sigma A}$ in $[\Alg]$. Define $\beta:t(\Gamma A)[t]\to (\Sigma A)[t]\tensor[_{\ev_1}]{\times}{_{\ev_1}}t(\Sigma A)[t]$ by $\beta(u)=(q_A(u)(1-t),0)$. The following diagram exhibits a strong morphism of extensions $\scrU_{\Sigma A}\to \scrP_{\Sigma A}$ extending the identity of $\Sigma A$:
    \[\xymatrix@C=4em{
    \scrU_{\Sigma A}\ar[d] & J\Sigma A\ar[d]_-{\chi}\ar[r] & T\Sigma A\ar[d]\ar[r] & \Sigma A\ar@{=}[d] \\
    (\scrE, s)\ar[d] & I\ar[d]_-{\beta}\ar[r] & t(\Gamma A)[t]\ar[d]^-{\beta}\ar[r] & \Sigma A\ar@{=}[d] \\
    \scrP_{\Sigma A} & (\Sigma A)^{S^1}_1\ar[r] & P(0,\Sigma A)_1\ar[r] & \Sigma A
    }\]
It follows that $\rho_{\Sigma A}$ equals the composite $J\Sigma A\overset{\chi}{\to}I\overset{\beta}{\to}(\Sigma A)^{S^1}_1$. To prove that $\rho_{\Sigma A}$ is homotopic to $\psi=(0, q_A)\circ \chi$, it suffices to show that $\beta, (0,q_A):I\to (\Sigma A)^{S^1}_1$ are homotopic, and the latter is easily verified (see for example the last part of the proof of \cite{loopstho}*{Lemma 4.3}). This proves our claim that $J(\psi)=J(\rho_{\Sigma A})$ in $[\Alg]$. Now, the commutativity of \eqref{eq:halg} implies that we have
\[x_{A^{S^1}_1}\circ J(\rho_{\Sigma A})=[\rho_{M_\infty A}\circ J(x_A)]^{-1}\in [J^2\Sigma A, (M_\infty A)^{S^1}_\bullet].\]
Here, the superscript $-1$ is the inverse for the group structure on $[J^2\Sigma A, (M_\infty A)^{S^1}_\bullet]$; see \cite{htpysimp}*{Example 3.12}. This finishes the proof.
\end{proof}


\begin{lem}
    Let $A\in\Alg$, $n\geq 0$ and $k\in \Z$. Then the identity of $J^nA$ induces an isomorphism
    \begin{equation}\label{eq:idJn}\langle\id_{J^nA}\rangle:(A,k)\to (J^nA,k-n)\end{equation}
    in $\ks$ that is natural in $A$ with respect to algebra homomorphisms.
\end{lem}
\begin{proof}
    By \cite{loopstho}*{Lemma 7.2} we have a natural isomorphism
    \[\langle\id_{J^nA}\rangle:(A,k)\to (J^nA,k-n)\]
    in $\fk$ that is natural in $A$ with respect to algebra homomorphisms. Upon applying $t_s\circ t_f$ we obtain the desired isomorphism \eqref{eq:idJn}.
\end{proof}

\begin{lem}\label{lem:idSn}
    Let $A\in\Alg$, $n,r\geq 0$ and $k\in\Z$. Then the identity of $A^{S^n}_r$ induces an isomorphism
    \begin{equation}\label{eq:idSn}\langle\id_{A^{S^n}_r}\rangle:(A^{S^n}_r,k-n)\to (A,k)\end{equation}
    in $\ks$ that is natural in $A$ with respect to algebra homomorphisms.
\end{lem}
\begin{proof}
    By \cite{loopstho}*{Lemma 7.8} we have a natural isomorphism
    \[\langle\id_{A^{S^n}_r}\rangle:(A^{S^n}_r,k-n)\to (A,k)\]
    in $\fk$ that is natural in $A$ with respect to algebra homomorphisms. Upon applying $t_s\circ t_f$ we obtain the desired isomorphism \eqref{eq:idSn}.
\end{proof}

\begin{defi}
    Let $A\in\Alg$. For $n\geq 0$ we will define an isomorphism
    \[\xymatrix{\epsilon^n_A:(\Sigma^n A, k)\ar[r] & (A, k-n)}\]
    in $\ks$ that is natural in $A$ with respect to algebra homomorphisms. Let $\epsilon^0_A$ be the identity of $(A,k)$ and let $\epsilon^1_A$ be the following composite in $\ks$:
    \[\xymatrix@C=4em{
    (\Sigma A, k)\ar[r]^-{\langle \id_{J\Sigma A}\rangle} & (J\Sigma A, k-1)\ar[r]^-{(x_A)_*} & (M_\infty A, k-1)\ar[r]^-{(\iota_A)^{-1}_*} & (A, k-1)
    }\]
    It is clear that $\epsilon^1_A$ is a natural isomorphism since each of the morphisms in the above composition are. Suppose now that we have defined $\epsilon^n_A$ for $n\geq 1$. Let $\epsilon^{n+1}_A$ be the following composite in $\ks$:
    \[\xymatrix@C=4em{
    (\Sigma^{n+1}A, k)\ar[r]^-{\epsilon^1_{\Sigma^n A}} & (\Sigma^n A, k-1)\ar[r]^-{\epsilon^n_{A}} & (A, k-1-n)
    }\]
    This defines $\epsilon^n_A$ for every $n\geq 0$.
\end{defi}

\begin{lem}\label{lem:epspq}
    Let $A\in\Alg$. For any $p,q\geq 0$ we have $\epsilon^{p+q}_A=\epsilon^p_A\circ \epsilon^q_{\Sigma^pA}$.
\end{lem}
\begin{proof}
    It follows from an easy induction on $q$. The base case $q=1$ holds by definition of $\epsilon^{p+1}_A$.
\end{proof}

\begin{lem}\label{lem:epsalpha}
    Let $A\in\Alg$. Recall from Definition \ref{defi:alpha} that we have isomorphisms $\alpha^{n,1}_A:(\Sigma^{1+n}JA,0)\to(\Sigma^nA,0)$ in $\ks$ for any $n\geq 0$. Then the following square in $\ks$ commutes for all $n\geq 0$:
    \[\xymatrix{
        (\Sigma^{1+n}JA, 0)\ar[r]^-{\alpha^{n,1}_A}\ar[d]_-{\epsilon^n_{\Sigma JA}} & (\Sigma^nA,0)\ar[d]^-{\epsilon^n_{A}} \\
        (\Sigma JA, -n)\ar[r]^-{\alpha^{0,1}_A} & (A, -n)
    }\]
\end{lem}
\begin{proof}
    The following diagram in $\ks$ is commutative by naturality of $\epsilon^n_?$:
    \[\xymatrix{
    (\Sigma^{1+n}JA, 0)\ar[d]_-{\epsilon^n_{\Sigma JA}}\ar[r]^-{(c_A)^{-1}_*}_-{\cong} & (\Sigma^nJ\Sigma A, 0)\ar[d]_-{\epsilon^n_{J\Sigma A}}\ar[r]^-{(x_A)_*} & (\Sigma^nM_\infty A, 0)\ar[d]^-{\epsilon^n_{M_\infty A}}\ar[r]^-{(\iota_A)^{-1}_*}_-{\cong} & (\Sigma^nA,0)\ar[d]^-{\epsilon^n_A} \\
    (\Sigma JA, -n)\ar[r]^-{(c_A)^{-1}_*}_-{\cong} & (J\Sigma A,-n)\ar[r]^-{(x_A)_*} & (M_\infty A, -n)\ar[r]^-{(\iota_A)^{-1}_*}_-{\cong} & (A,-n)
    }\]
    By Definition \ref{defi:alpha}, the composite of the morphisms in the top row is $\alpha^{n,1}_A$ and the composite of the morphisms in the bottom row is $\alpha^{0,1}_A$.
\end{proof}

\begin{lem}\label{lem:tribeta}
    The following diagram in $\ks$ commutes for all $A\in\Alg$ and $n\geq 0$:
    \begin{equation}\label{eq:tribeta}\begin{gathered}\xymatrix@C=4em{
    (\Sigma^n A^{S^n}, 0)\ar[d]_-{\beta^{0,n}_A}\ar[r]^-{(-1)^n\epsilon^n_{A^{S^n}}} & (A^{S^n}, -n)\ar@/^1pc/[dl]^-{\langle \id_{A^{S^n}}\rangle} \\
    (A,0) &
    }\end{gathered}\end{equation}
\end{lem}
\begin{proof}
    We proceed by induction on $n$. For $n=0$ there is nothing to prove. Let us prove that \eqref{eq:tribeta} commutes for $n=1$. The following diagram shows a strong morphism of extensions $\scrU_{\Sigma A}\to \scrP_{\Sigma A}$ extending the identity of $\Sigma A$:
    \[\xymatrix@C=4em@R=1.5em{
    \scrU_{\Sigma A}\ar[d] & J\Sigma A\ar[d]_-{c_A}\ar[r] & T\Sigma A\ar[d]\ar[r] & \Sigma A\ar@{=}[d] \\
    \Sigma(\scrU_A)\ar[d] & \Sigma JA\ar[d]_-{\Sigma(\rho_A)}\ar[r] & \Sigma TA\ar[d]\ar[r] & \Sigma A\ar@{=}[d] \\
    \Sigma(\scrP_A) & \Sigma A^{S^1}\ar[r] & \Sigma PA\ar[r] & \Sigma A
    }\]
    It follows that $\rho_{\Sigma A}=\Sigma(\rho_A)\circ c_A:J\Sigma A \to\Sigma A^{S^1}$. Now consider the following diagram in $\ks$, where all the morphisms are isomorphisms:
    \[\xymatrix@C=3.8em@R=1.5em{
    (\Sigma A^{S^1}, 0)\ar[d]_-{(\rho_{\Sigma A})_*^{-1}}\ar[r]^-{\langle \id_{J(\Sigma A^{S^1})}\rangle} & (J(\Sigma A^{S^1}), -1)\ar[r]^-{-(x_{A^{S^1}})_*} & (M_\infty A^{S^1}, -1)\ar[r]^-{(\iota_{A^{S^1}})_*^{-1}}\ar@/^1pc/[ddll]_-{\langle \id_{M_\infty A^{S^1}}\rangle} & (A^{S^1}, -1)\ar@/^1pc/[dddlll]^-{\langle \id_{A^{S^1}}\rangle} \\
    (J\Sigma A, 0)\ar[d]_-{(x_A)_*}\ar@/_1pc/[ur]^-{\langle J(\rho_{\Sigma A})\rangle} & & & \\
    (M_\infty A, 0)\ar[d]_-{(\iota_A)_*^{-1}} & & & \\
    (A, 0) & & &
    }\]
    The small triangle in the upper left corner commutes by definition of the composition law in $\fk$. The lower rectangle commutes by naturality of the isomorphism $(A,0)\cong (A^{S^1}, -1)$. The rectangle in the middle commutes by Lemma \ref{lem:extdoble} (recall from \cite{loopstho}*{Lemma 7.4} that $\langle \id_{M_\infty A^{S^1}}\rangle$ is the inverse of $\langle \rho_{M_\infty A}\rangle$). It follows that the outer square commutes, and this proves the result for $n=1$. Indeed, the composite of the morphisms in the top row equals $-\epsilon^1_{A^{S^1}}$ and the composite of the morphisms in the left column equals $\beta^{0,1}_A$.

    Suppose that \eqref{eq:tribeta} commutes for $n\geq 1$ and let us prove that it also commutes for $n+1$. Consider the following diagram in $\ks$ where all the morphisms are isomorphisms:
    \[\xymatrix@R=1.5em@C=5.5em{
    (\Sigma^{n+1}A^{S^{n+1}}, 0)\ar[d]_-{(\mu^{n,1}_A)_*^{-1}}\ar[r]^-{(-1)^n\epsilon^n_{\Sigma A^{S^{n+1}}}} & (\Sigma A^{S^{n+1}}, -n)\ar[d]_-{(\mu^{n,1}_A)_*^{-1}}\ar[r]^-{-\epsilon^1_{A^{S^{n+1}}}} & (A^{S^{n+1}}, -n-1)\ar[d]^-{(\mu^{n,1}_A)_*^{-1}} \\
    (\Sigma^{n+1}(A^{S^n})^{S^1}, 0)\ar[d]_-{\beta^{n,1}_{A^{S^n}}}\ar[r]^-{(-1)^n\epsilon^n_{\Sigma (A^{S^n})^{S^1}}} & (\Sigma (A^{S^n})^{S^1},-n)\ar[d]_-{\beta^{0,1}_{A^{S^n}}}\ar[r]^-{-\epsilon^1_{(A^{S^n})^{S^1}}} & ((A^{S^n})^{S^1}, -n-1)\ar@/^1pc/[dl]^-{\langle \id_{(A^{S^n})^{S^1}}\rangle} \\
    (\Sigma^nA^{S^n}, 0)\ar[r]^-{(-1)^n\epsilon^n_{A^{S^n}}}\ar[d]_-{\beta^{0,n}_A} & (A^{S^n}, -n)\ar@/^1pc/[dl]^-{\langle \id_{A^{S^n}}\rangle} & \\
    (A,0) & &
    }\]
    The three squares commute by naturality of $\epsilon^n_?$ and the two triangles commute by the inductive hypothesis. Is follows that the outer diagram commutes, proving that \eqref{eq:tribeta} commutes for $n+1$. Indeed, the composite of the morphisms in the top row equals $(-1)^{n+1}\epsilon^{n+1}_{S^{S^{n+1}}}$ by Lemma \ref{lem:epspq}, the composite of the morphisms in the left column equals $\beta^{0, n+1}_A$ by Lemma \ref{lem:betapq} and the composite of the rightmost vertical morphism followed by the diagonal ones equals $\langle \id_{A^{S^{n+1}}}\rangle$ by definition of the composition law in $\fk$.
\end{proof}

\begin{lem}\label{lem:trialpha}
    The following diagram in $\ks$ commutes for all $A\in\Alg$ and $n\geq 0$:
    \begin{equation}\label{eq:trialpha}\begin{gathered}\xymatrix@C=4em{
    (\Sigma^n J^nA, 0)\ar[d]_-{\alpha^{0,n}_A}\ar[r]^-{(-1)^n\epsilon^n_{J^nA}} & (J^nA, -n) \\
    (A,0)\ar@/_1pc/[ur]_-{\langle \id_{J^nA}\rangle} &
    }\end{gathered}\end{equation}
\end{lem}
\begin{proof}
    We proceed by induction on $n$. For $n=0$ there is nothing to prove. For $n=1$, consider the following diagram in $\ks$ where all the morphisms are isomorphisms:
    \[\xymatrix@C=5em@R=1.5em{
    (\Sigma JA)\ar[d]_-{(\rho_A)_*}\ar[r]^-{-\epsilon^1_{JA}} & (JA, -1)\ar[d]^-{(\rho_A)_*} \\
    (\Sigma A^{S^1})\ar[d]_-{\beta^{0,1}_A}\ar[r]^-{-\epsilon^1_{A^{S^1}}} & (A^{S^1}, -1)\ar@/^1pc/[dl]^-{\langle \id_{A^{S^1}}\rangle} \\
    (A, 0)
    }\]
    The square commutes by naturality of $\epsilon^1_?$ and the triangle commutes by Lemma \ref{lem:tribeta}. It follows that the outer diagram commutes. The composite of the morphisms in the left column equals $\alpha^{0,1}_A$ by definition of $\beta^{0,1}_A$ and $\langle \id_{A^{S^1}}\rangle\circ (\rho_A)_*=\langle \rho_A\rangle$ is the inverse of $\langle \id_{JA}\rangle$ by \cite{loopstho}*{Lemma 7.2}. This proves that \eqref{eq:trialpha} commutes for $n=1$.
    
    Suppose \eqref{eq:trialpha} commutes for $n\geq 1$ and let us prove that it also commutes for $n+1$. Consider the following diagram in $\ks$:
    \[\xymatrix@C=5em{(\Sigma^{1+n}J^{n+1}A,0)\ar[r]^-{(-1)^n\epsilon^n_{\Sigma J^{n+1}A}}\ar[d]_-{\alpha^{n,1}_{J^nA}} & (\Sigma J^{n+1}A,-n)\ar[r]^-{-\epsilon^1_{J^{n+1}A}}\ar[d]^-{\alpha^{0,1}_{J^nA}} & (J^{n+1}A, -1-n) \\
    (\Sigma^n J^nA, 0)\ar[r]^-{(-1)^n\epsilon^n_{J^nA}}\ar[d]_-{\alpha^{0,n}_A} & (J^nA, -n)\ar@/_1pc/[ur]_-{\langle\id_{J^{n+1}A}\rangle} & \\
    (A,0)\ar@/_1pc/[ur]_-{\langle\id_{J^nA} \rangle} & &
    }\]
    Both triangles commute by the inductive hypothesis and the square commutes by Lemma \ref{lem:epsalpha}. The composite of the morphisms in the top row is $(-1)^{n+1}\epsilon^{n+1}_{J^{n+1}A}$ by Lemma \ref{lem:epspq} and the composite of the morphisms in the left column is $\alpha^{0, n+1}_A$ by Lemma \ref{lem:alphapq}. The result follows.
\end{proof}

\begin{lem}\label{lem:zzagkk}
    Let $\alpha\in\ks(A,B)$ and let $f:J^nA\to M_pM_\infty B^{S^n}_r$ be an algebra homomorphism representing $\alpha$. Let $\tilde{f}$ be the following composite in $\ks$:
    \begin{equation}\label{eq:tildef}\xymatrix{J^nA\ar[r]^-{f_*} & M_pM_\infty B^{S^n}_r\ar[r]^-{\iota_*^{-1}}_-{\cong} & M_\infty B^{S^n}_r \ar[r]^-{\iota_*^{-1}}_-{\cong} & B^{S^n}_r\ar[r]^-{\gamma_*^{-1}}_-{\cong} & B^{S^n}}\end{equation}
    Here, the isomorphism labelled $\gamma$ is induced by the last vertex map and the isomorphisms labelled $\iota$ are induced by upper-left corner inclusions into matrix algebras. Then the following diagram in $\ks$ commutes:
    \[\xymatrix{\Sigma^n J^nA\ar[d]_-{\Sigma^n\tilde{f}}\ar[r]^-{\alpha^{0,n}_A}_-{\cong} &  A\ar[d]^-{\alpha} \\
    \Sigma^nB^{S^n}\ar[r]^-{\beta^{0,n}_B}_-{\cong} & B}\]
\end{lem}
\begin{proof}
    By Lemmas \ref{lem:tribeta} and \ref{lem:trialpha} we may replace $\alpha^{0,n}_A$ by the composite
    \[\xymatrix@C=4em{(\Sigma^n J^nA,0)\ar[r]^-{\epsilon^n_{J^nA}}_-{\cong} & (J^nA,-n)\ar[r]^-{\langle \id_{J^nA}\rangle^{-1}}_-{\cong} & (A,0)}\]
    and $\beta^{0,n}_B$ by the composite
    \[\xymatrix@C=4em{(\Sigma^n B^{S^n},0)\ar[r]^-{\epsilon^n_{B^{S^n}}}_-{\cong} & (B^{S^n},-n)\ar[r]^-{\langle\id_{B^{S^n}}\rangle}_-{\cong} & (B,0)};\]
    note that the signs $(-1)^n$ cancel out. Consider the following diagram in $\ks$:
    \[\xymatrix@R=1.5em@C=5em{
    (\Sigma^nJ^nA,0)\ar[d]_-{f_*}\ar[r]^-{\epsilon^n}_-{\cong} & (J^nA,-n)\ar[d]_-{f_*}\ar[r]^-{\langle\id_{J^nA}\rangle^{-1}}_-{\cong} & (A,0)\ar[d]^-{f_*} \\
    (\Sigma^nM_pM_\infty B^{S^n}_r,0)\ar[r]^-{\epsilon^n}_-{\cong}\ar[d]_-{\iota_*^{-1}} & (M_pM_\infty B^{S^n}_r,-n)\ar[r]^-{\langle\id_{M_pM_\infty B^{S^n}}\rangle}_-{\cong}\ar[d]_-{\iota_*^{-1}} & (M_pM_\infty B,0)\ar[d]^-{\iota_*^{-1}} \\
    (\Sigma^nM_\infty B^{S^n}_r,0)\ar[r]^-{\epsilon^n}_-{\cong}\ar[d]_-{\iota_*^{-1}} & (M_\infty B^{S^n}_r,-n)\ar[r]^-{\langle\id_{M_\infty B^{S^n}}\rangle}_-{\cong}\ar[d]_-{\iota_*^{-1}} & (M_\infty B,0)\ar[d]^-{\iota_*^{-1}} \\
    (\Sigma^nB^{S^n}_r,0)\ar[r]^-{\epsilon^n}_-{\cong}\ar[d]_-{\gamma_*^{-1}} & (B^{S^n}_r,-n)\ar[r]^-{\langle\id_{B^{S^n}}\rangle}_-{\cong}\ar[d]_-{\gamma_*^{-1}} & (B,0)\ar@{=}[d] \\
    (\Sigma^n B^{S^n},0)\ar[r]^-{\epsilon^n}_-{\cong} & (B^{S^n},-n)\ar[r]^-{\langle\id_{B^{S^n}}\rangle}_-{\cong} & (B,0)
    }\]
    Since the composite of the vertical morphisms on the left column is $\Sigma^n\tilde{f}$ and the composite of the morphisms on the right column is $\alpha$, the result will follow if we prove that the outer square commutes. The squares on the left commute bu naturality of $\epsilon^n$. The upper-right square commutes by \cite{loopstho}*{Lemma 7.10}. The remaining squares commute by the naturality stated in Lemma \ref{lem:idSn}. This finishes the proof.
\end{proof}

\end{document}